\documentclass[11pt]{article}
\usepackage{appendix,latexsym,amsfonts,amsmath,amssymb,graphicx,hyperref,amsthm,authblk,comment,epstopdf}
\usepackage{graphicx}
\usepackage{pinlabel}

\newcommand{\EE}{\mathbb{E}}
\newcommand{\PP}{\mathbb{P}}

\newcommand{\R}{\mathbb{R}}
\newcommand{\C}{\mathbb{C}}

\newcommand{\HH}{\mathbb{H}}
\newcommand{\N}{\mathbb{N}}
\newcommand{\D}{\mathbb{D}}

\newcommand{\Z}{\mathbb{Z}}
\newcommand{\Ree}{\mbox{Re}\,}

\newcommand{\pa}{\partial}

\newcommand{\F}{{\cal F}}
\newcommand{\no}{\noindent}

\newcommand{\BGE}{\begin{equation}}
\newcommand{\BGEN}{\begin{equation*}}
\newcommand{\EDE}{\end{equation}}
\newcommand{\EDEN}{\end{equation*}}

\newcommand  \MC{{{\rm Cont}_d}}
\newcommand \mcon { \MC}

\newcommand{\Half}{{\mathbb H}}

\def\eps{\varepsilon}

\def\til{\widetilde}
\def\ha{\widehat}
\def\sem{\setminus}
\def\lin{\overline}

 \DeclareMathOperator{\rad}{rad}
 \DeclareMathOperator{\diam}{diam}
\DeclareMathOperator{\dist}{dist} 
\DeclareMathOperator{\hcap}{hcap}

  \DeclareMathOperator{\Imm}{Im}

\newtheorem{Lemma}{Lemma}[section]
\newtheorem{Theorem}{Theorem}[section]

\numberwithin{equation}{section}

\begin{document}

\title{Higher moments of the natural parameterization for SLE curves}
\author{Mohammad A. Rezaei}
\author{Dapeng Zhan\thanks{Research partially supported by NSF grant  DMS-1056840 and Sloan fellowship.}}
\affil{Michigan State University}
\renewcommand\Authands{ and }
\maketitle

\begin{abstract}
In this paper, we will show that the higher moments of the natural parametrization of $SLE$ curves in any bounded domain in the upper half plane is finite. We prove this by estimating the probability that an $SLE$ curve gets near $n$ given points.

\end{abstract}

\section{Introduction}   \label{Introsec}

A number of measures arise from statistical physics are believed to have conformally invariant scaling limits. In \cite{Sch}, a one-parameter family of measures on non-self-crossing curves in the upper half plane, called (chordal) Schramm-Loewner evolution ($SLE_\kappa$) is defined. Here we only work with chordal version so we omit chordal. By conformal invariance, it is extended to other simply connected domains. Later, it was shown that $SLE$ describes the limits of a number of models from physics so answering the question of conformal invariance for them. These models include  loop-erased random walk for $\kappa=2$ \cite{LSW}, Ising interfaces for $\kappa=3$ and $\kappa=16/3$ \cite{Smir2}, harmonic explorer for $\kappa=4$ \cite{SS}, percolation interfaces for $\kappa=6 $ \cite{Smir1}, and uniform spanning tree Peano curves for $\kappa=8$ \cite{LSW}.

In order to define $SLE$, Schramm used {\em capacity parametrization}. We will see the definition of $SLE$ as well as capacity parametrization in the next section. Capacity parametrization comes from Loewner evolution and makes it easy to analyze $SLE$ curves by Ito's calculus. In all the physical models that we have above, in order to show the convergence, we have to first parametrize them with discrete version of the capacity and then prove the convergence to $SLE$. This parametrization is very different from the {\em natural} parametrization that we have for them which is just the length of the curve.

In order to prove the same results with the natural parametrization, we need to define a natural length for $SLE$ curves. In \cite{Bf}, it is proved that the Hausdorff dimension of $SLE_\kappa$ is $d=\min\{2,1+\frac{\kappa}{8}\}$. In \cite{LS}, the authors conjectured that the Minkowski content of SLE should exist. They defined the natural parametrization in a different way using Doob-Meyer decomposition and proved the existence for  $\kappa <5.021...$. Moreover, they conjectured  that the natural length of $SLE$ can be defined in terms of $d$-dimensional Minkowski content. Here is how it is defined (see \cite{LR} for more details). Let
\[  \mcon(\gamma[0,t];r) =  r^{d-2}
                     {\rm Area} \left\{z: \dist(z,\gamma[0,t]) \leq r
                    \right\}. \]
                    Then the
$d$-dimensional content
 is
\begin{equation}  \label{minkcont}
              \mcon(\gamma[0,t]) =
               \lim_{r \rightarrow 0} \mcon(\gamma[0,t];r)  ,
                     \end{equation}
provided that the limit exists. If $\kappa>8$ the curve is space filling and $d=2$ so this is just the area and the problem is trivial. For $k<8$, the existence was shown in \cite{LR}. We assume for the purpose of this paper that $\kappa <8$. We call this parametrization, natural length or length from now on. Also a number of properties of the natural length were studied in \cite{LR}. For example the authors computed the first and second moments of the ``natural length''. Basically, this function is the appropriate scaled version of the probability that $SLE$ hits given point(s). Precisely, the $n$-point Green's function at $z_1,\cdots,z_n$ is
\BGE   G(z_1,\dots,z_n) = \lim_{r_1,\dots,r_n \rightarrow 0} \prod_{k=1}^n r_k^{d-2}  \, \PP\Big [\bigcap_{k=1}^n\{\dist(z_k,\gamma) \leq r_k\}\Big ], \label{mlti-green}\EDE
provided that the limit exists. The covariance rule of the Green's function is obvious, that is, if $F$ maps $(\HH;0,\infty)$ conformally onto $(D;w_1,w_2)$, then
\BGE
G_{(D;w_1,w_2)}(z_1,\dots,z_n)=|(F^{-1})'(z)|^{2-d}G_{(\HH;0,\infty)}(F^{-1}(z_1),\dots,F^{-1}(z_n)), \label{green-D}
\EDE
if the Green's function at either side exists. Here we use $G_{(D;w_1,w_2)}$ to denote the Green's function for $SLE_\kappa$ in $D$ from $w_1$ to $w_2$.

It is proved in \cite{LW} that a modified version of $1$-point and $2$-point Green's function using conformal distance instead of distance exists. In \cite{LR}, the authors prove the above limit exists for $n=1,2$.
Lawler and Werness mentioned in \cite{LW} that the argument can be generalized to define higher order Green's function. So they conjectured the existence of multi-point Green's function. For $n=1$ the exact formula is given in \cite{LR} which is
 \begin{equation}  \label{green}
     G(z)=G_{(\HH;0,\infty)}(z)
   = C |z|^{d-2}  \, \sin^{\frac
  \kappa 8 + \frac 8{\kappa} -2} (\arg z)  =C \Imm(z)^{d-2}
    \, \sin^{8/\kappa -1 }(\arg z) ,
\end{equation}
where $C=C_\kappa>0$ is an unknown constant. In arbitrary domains the exact formula of the $1$-point Green's function can be found by the covariance rule.

We now state the main theorems of this paper. Throughout, we fix $\kappa\in(0,8)$, the following constants depending on $\kappa$:
\[
d=1+\frac{\kappa}{8}, \qquad \alpha=\frac{8}{\kappa}-1.
\]
We will use $C$ to denote an arbitrary positive constant that depends only on $\kappa$, whose value may vary from one occurrence to another. If we allow $C$ to depend on $\kappa$ and another variable, say $n$, then we will use $C_n$.
We introduce a family of functions. For $y\ge 0$, define $P_y$ on $[0,\infty)$ by
$$ P_y(x)=\left\{\begin{array}{ll} y^{\alpha-(2-d)}x^{2-d},&  x\le y;\\  x^\alpha,& x\ge y.
\end{array}\right. $$
Since $\alpha\ge 2-d>0$, if $0\le x_1< x_2$, then
\BGE
\frac{x_1^\alpha}{x_2^\alpha}\le \frac{P_y(x_1)}{P_y(x_2)}\le \frac{x_1^{2-d}}{x_2^{2-d}}.\label{P-compare}\EDE

The first main theorem is:
\begin{Theorem}  \label{mainthm2}
  Let $z_0,\dots,z_n$ be distinct points on $\lin\HH$ such that $z_0=0$. Let $y_k=\Imm z_k\ge 0$ and $l_k=\dist(z_k,\{z_j:0\le j<k\})$, $1\le k\le n$. Let $r_1,\dots,r_n>0$. Let $\gamma$ be an $SLE_\kappa$ curve in $\HH$ from $0$ to $\infty$. Then there is $C_n<\infty$ depending only on $\kappa$ and $n$ such that
  $$\PP[\dist(\gamma,z_k)\le r_k,1\le k\le n]\le C_n\prod_{k=1}^n \frac{P_{y_k}(r_k\wedge l_k)}{P_{y_k}(l_k)}.$$
\end{Theorem}

\vskip 4mm
The second main theorem answers a question in \cite{LR}.

\begin{Theorem} \label{mainthm}
If $\gamma$ is an $SLE$ curve from 0 to $\infty$ in $\HH$, then for any bounded $D \subset \HH$, we have
\[
\EE[\mcon(\gamma \cap D)^n]<\infty,\quad n\in\N.
\]
\end{Theorem}

\no{\bf Remarks.}
\begin{enumerate}

\item The quantity on the right-hand side of the formula in Theorem \ref{mainthm2} depends on the order of the points $z_1,\dots,z_n$. However, if $r_j$'s are sufficiently small, say, $r_j<\dist(z_j,\{z_0,\dots,z_n\}\sem\{z_j\})$, then if we exchange any pair of consecutive points, i.e., $z_k$ and $z_{k+1}$, then the new quantity is no more than $C$ times the old quantity, where $C>0$ depends only on $\kappa$. Thus, if we permute those $n$ points, the quantity will increase at most $C^{n^2}$ times.

\item An immediate consequence of Theorem \ref{mainthm2} is that the right-hand side of (\ref{mlti-green}), with $\lim$ replaced by $\limsup$, is finite.

  \item In fact, Theorem \ref{mainthm2} implies an upper bound of the Green's function $G(z_1,\cdots,z_n)$ for the above $\gamma$, if it exists. That is
$$ G(z_1,\dots,z_n)\le C_n\prod_{k=1}^n \frac{y_k^{\alpha-(2-d)}}{P_{y_k}(l_k)}.$$ 
A natural question to ask is whether the reverse inequality also holds (with smaller $C_n$). The answer is yes if $n\le 2$. In the case $n=1$, the right-hand side is $C\frac{y^{\alpha-(2-d)}}{|z|^\alpha}$, which agrees with the right-hand side of (\ref{green}). In the case $n=2$, the right-hand side is comparable to a sharp estimate of the $2$-point Green's function given in \cite{LR2} up to a constant. Thus, we expect that it holds for all $n\in\N$.

\item We guess that one can show $\EE[e^{\lambda \mcon(\gamma \cap D)}]<\infty$ for some $\lambda>0$ in any bounded domain $D$. This is nice because we can study natural length by its moment generating function. One way to prove it is to prove a similar bound for ordered multi-point Green's function but with $C^n$ instead of $C_n$. See \cite{LW} for the definition of ordered Green's function.

\item If the Green's function $G(z_1,\cdots,z_n)$ exits, the left-hand side of the displayed formula in Theorem \ref{mainthm} equals to $\int_{D^n} G(z_1,...,z_n) dA(z_1)...dA(z_n)$. 
\item Theorem \ref{mainthm2} also provides an upper bound for the boundary Green's function, which is the scaled version of the probability that $SLE$ hits given boundary point(s). The scaling exponent will be $\alpha$ instead of $2-d$ so that the Green's function does not vanish. To be more precise, for the above $\gamma$, the boundary Green's function at $x_1,\dots,x_n\in\R\sem\{0\}$ is
    \BGE \tilde G(x_1,\dots,x_n)=\lim_{r_1,\dots,r_n \rightarrow 0} \prod_{k=1}^n r_k^{-\alpha}  \, \PP\Big [\bigcap_{k=1}^n\{\dist(x_k,\gamma) \leq r_k\}\Big ],\label{green-bdry}\EDE
    provided that the limit exists. Lawler recently proved in \cite{Law4} that the $1$-point and $2$-point boundary Green's function exist, and gave good estimates of these functions. Using Theorem \ref{mainthm2}, we can derive the following conclusions. First, the right-hand side of (\ref{green-bdry}), with $\lim$ replaced by $\limsup$, is finite. This result may help us to prove the existence of multi-point boundary Green's functions for $SLE$. Second, if $\tilde G(x_1,\cdots,x_n)$ exits, then $\tilde G(x_1,\dots,x_n)\le C_n\prod_{k=1}^n l_k^{-\alpha}$, where $l_k=\min_{0\le j<k} |x_k-x_j|$ with $x_0=0$. Similarly, we get upper bounds for mixed Green's functions, where some points lie on the boundary, and others lie in the interior.
\end{enumerate}


The organization of the rest of the paper goes as follows. In the next section we review the definition of $SLE$ and some fundamental estimates for $SLE$. In the third section, we will prove two main lemmas. At the end, we will prove the two main theorems.

\section*{Acknowledgement}
Both authors thank Gregory Lawler for his valuable comments on this project, and an anonymous referee for very helpful comments on an earlier version of this paper. Dapeng Zhan acknowledges the support from the National Science Foundation under the grant DMS-1056840 and the support from the Alfred P.\ Sloan Foundation.

\section{Preliminaries}  \label{Prelimsec}
\subsection{Definition of $SLE$}
In this subsection we review the definition of $SLE$ and its basic properties. See \cite{Law1,Law2,LW,LR} for more details.

A bounded set $K\subset\HH=\{x+iy: y > 0\}$ is called an $\HH$-hull if $\HH\sem K$ is a simply connected domain, and the complement $\HH\sem K$ is called an $\HH$-domain. For every $\HH$-hull $K$, there is a unique conformal map $g_K$ from $\HH\sem K$ onto $\HH$ that satisfies
$$g_K(z)=z+\frac{c}z+O(|z|^{-2}),\quad |z|\to\infty$$
for some $c\ge 0$. The number $c$ is called the half plane capacity of $K$, and is denoted by $\hcap(K)$.

Suppose  that $\gamma:(0,\infty) \rightarrow \Half $
is a simple curve with $\gamma(0+) \in \R$ and $\gamma(t) \rightarrow \infty$ as $t \rightarrow
\infty$. Then for each $t$, $K_t:=\gamma(0,t]$ is an $\HH$-hull. Let $g_t=g_{K_t}$ and $a(t)=\hcap(K_t)$. We can reparameterize the curve such that $a(t)=2t$.  
Then $g_t$ satisfies the
  {\em  (chordal) Loewner equation}
\begin{equation}  \label{loew}
       \pa_t g_t(z) = \frac{2}{g_t(z) - V_t} , \quad
   g_0(z) = z,
\end{equation}
where $V_t: = \lim_{\HH\sem K_t\ni z\to \gamma(t)} g_t(z)$ is a continuous real-valued function.

Conversely, one can start with a continuous real-valued function $V_t$  and define $g_t$ by \eqref{loew}.
For $z\in \Half \setminus \{0\}$, the function $t \mapsto g_t(z)$ is
well defined up to a blowup time $T_z$, which could be $\infty$.
The evolution then generates an increasing family of $\HH$-hulls defined by
$$K_t=\{z\in\HH:T_z>t\},\quad 0\le t<\infty,$$
with $g_t=g_{K_t}$ and $\hcap(K_t)=2t$ for each $t$. One may not always get a curve from the evolution.

The {\em (chordal) Schramm-Loewner evolution ($SLE_\kappa$) (from $0$
to $\infty$ in $\Half$)} is the solution to \eqref{loew}
where  $V_t = \sqrt{\kappa} B_t$, where $\kappa>0$ and $B(t)$ is a standard Brownian motion.
It is shown in \cite{RS,LSW} that the limits
$$\gamma(t)=\lim_{\HH\ni z\to V_t} g_t^{-1}(z),\quad 0\le t<\infty,$$
exist, and give a continuous curve $\gamma$ in $\lin\HH$ with $\gamma(0)=0$ and $\lim_{t\to\infty}\gamma(t)=\infty$. Only in the case $\kappa\le 4$, the curve is simple and stays in $\HH$ for $t>0$, and we recover the previous picture. For other cases, $\gamma$ is not simple, and $H_t:=\HH\sem K_t$ is the unbounded component of $\HH\sem \gamma(0,t]$.


We can define $SLE_\kappa$ in other simply connected domains using conformal maps. Roughly speaking, $SLE_\kappa$ in a simply connected domain $D\subsetneqq\C$ is the image of the above $\gamma$ under a conformal map $F$ from $\HH$ onto $D$. However, since $\gamma$ in fact lies in $\lin\HH$ instead of $\HH$, the rigorous definition requires some regularity of $D$. For simplicity, we assume that $\pa D$ is locally connected and call such domain $D$ regular. This ensures that any conformal map $F$ from $\HH$ onto $D$ has a continuous extension to $\lin\HH$, and so $F\circ \gamma$ is a continuous curve in $\lin D$.

Now we state the definition. Let $D$ be a regular simply connected domain, and $w_0,w_\infty$ be distinct prime ends (c.f.\ \cite{Law1}) of $D$. Let $F: \Half
\rightarrow D$ be a conformal transformation of $\Half$ onto $D$ with
$F(0) = w_0, F(\infty) = w_\infty$.  Then
$  \tilde \gamma:= F \circ \gamma$ is called an $SLE_\kappa$ curve in $D$ from $w_0$ to $w_\infty$.
Although such $F$ is not unique, the definition is unique up to a linear time change.

Now we state the important {\em Domain Markov Property} (DMP) of $SLE$. Let $D$ be a regular simply connected domain with prime ends $w_0\ne w_\infty$, and $\gamma$ an $SLE_\kappa$ curve in $D$ from $w_0$ to $w_\infty$. For each $t_0\ge 0$, let $D_{t_0}$ be the connected component of $\HH\sem \gamma(0,t_0]$ which is a neighborhood of $w_\infty$ in $D$, and $\gamma^{t_0}(t)=\gamma(t_0+t)$, $0\le t<\infty$. Let $T$ be any stopping time w.r.t.\ $\gamma$. Then conditioned on $\gamma(0,T]$ and the event $\{T<\infty\}$, a.s.\ $\gamma(T)\in\pa D_T$ determines a prime end of $D_T$, and $\gamma^T$ has the distribution of $SLE_\kappa$ in $D_T$ from (the prime end determined by) $\gamma(T)$ to $w_\infty$.

\subsection{Crosscuts}
Let $D$ be a simply connected domain. A simple curve $\rho:(a,b)\to D$ is called a crosscut in $D$ if $\lim_{t\to a^+}\rho(t)$ and $\lim_{t\to b^-}\rho(t)$ both exist and lie on $\pa D$. We emphasize that by definition the end points of $\rho$ do not belong to $\rho$, and so $\rho$ completely lies in $D$. It is well known (c.f.\ \cite{Pom-bond}) that as $t\to a^+$ or $t\to b^-$, $\rho(t)$ tends to a prime end of $D$. We say that these two prime ends are determined by $\rho$. Thus, if $f$ maps $D$ conformally onto $\D$, then $f(\rho)$ is a crosscut in $\D$. So we see that $D\sem\rho$ has exactly two connected components.

For the ease of labeling the two components of $D\sem \rho$, we introduce the following symbols. Let $K$ be any subset of $\C$ such that $K\cap D$ is a relatively closed subset of $D$, and let $S$ be a connected subset of $D\sem K$. We use $D(K;S)$ to denote the connected component of $D\sem K$ which is a neighborhood of $S$ in $D$; and let $D^*(K;S)=D\sem (K\cup D(K;S))$, which is the union of components of $D\sem K$ other than $D(K;S)$. For example, $D(K;z_1)\ne D(K;z_2)$ means that $z_1$ and $z_2$ are separated in $D$ by $K$. If $\rho$ and $\eta$ are disjoint crosscuts in $D$. Then $D\sem\rho=D(\rho;\eta)\cup D^*(\rho;\eta)$ and $D\sem\eta=D(\eta;\rho)\cup D^*(\eta;\rho)$; and we have $D^*(\rho;\eta)\subset D(\eta;\rho)$ and $D^*(\eta;\rho)\subset D(\rho;\eta)$.

The symbols $D(K;S)$ and $D^*(K;S)$ also make sense if $S$ is a prime end of $D$ such that $D\sem K$ is a neighborhood of $S$ in $D$. If $D$ is an $\HH$-domain, and $S$ is the prime end $\infty$, then we omit the $\infty$ in $D(K;\infty)$ and $D^*(K;\infty)$. For example, for the $SLE_\kappa$ curve $\gamma$ in $\HH$ from $0$ to $\infty$, the corresponding $\HH$-hull $K_t$ satisfies that $\HH\sem K_t=\HH(\gamma(0,t])$.

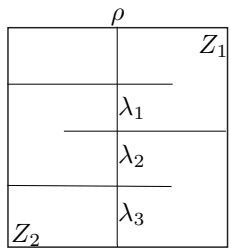
\begin{figure}

\labellist
\small
\pinlabel $Z_2$ at 210 18
\pinlabel $Z_1$ at 405 215
\pinlabel $\lambda_3$ at 320 40
\pinlabel $\lambda_2$ at 320 100
\pinlabel $\lambda_1$ at 320 150
\pinlabel $\rho$ at 305 245
\endlabellist

   \centering
	\includegraphics[width=3in]{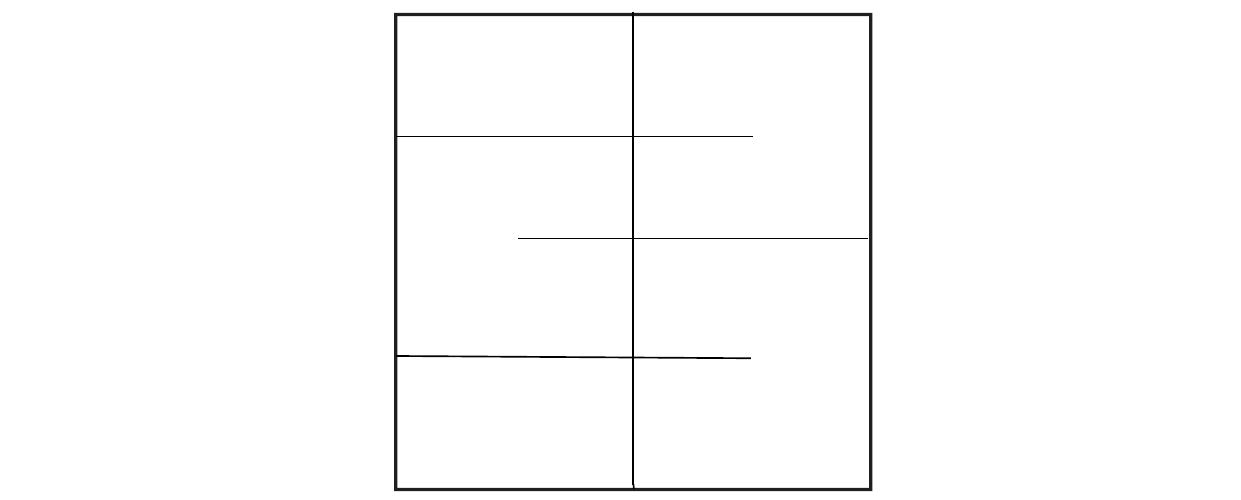}
	\caption{ This figure illustrates the situation of Lemma \ref{first-crosscut}. Here $\til D$ is the square and $D$ is the comb domain. The $\lambda_1$, $\lambda_2$, $\lambda_3$ are sub-crosscuts of $\rho$ in $D$ that separate $Z_1$ from $Z_2$. Among them $\lambda_1$ is the crosscut given by the lemma. }
\end{figure}

\begin{Lemma}
Let $D\subset \til D$ be two simply connected domains. Let $\rho$ be a Jordan curve in $\til D$, which intersects $\pa D$, or a crosscut in $\til D$. Let $Z_1$ and $Z_2$ be two connected subsets or prime ends of $\til D$ such that $\til D(\rho;Z_j)$, $j=1,2$, are well defined and not equal. In other words, $\til D\sem \rho$ is a neighborhood of both $Z_1$ and $Z_2$ in $D$, and $Z_1$ is disconnected from $Z_2$ in $\til D$ by $\rho$. Suppose $D$ is a neighborhood of both $Z_1$ and $Z_2$ in $\til D$.
Let $\Lambda$ denote the set of connected components of $\rho\cap D$. Then there is a unique $\lambda_1\in\Lambda$ such that $D(\lambda_1;Z_1)\ne D(\lambda_1;Z_2)$, and if $D(\lambda;Z_1)\ne D(\lambda;Z_2)$ for some $\lambda\in \Lambda$, then $D(\lambda_1;Z_1)\subset D(\lambda;Z_1)$ and $D(\lambda_1;Z_2)\supset D(\lambda;Z_2)$. \label{first-crosscut}
\end{Lemma}

\no{\bf Remark.}
Every $\lambda\in \Lambda$ is a crosscut in $D$. We call the $\lambda_1$ given by the lemma the first sub-crosscut of $\rho$ in $D$ that disconnects $Z_1$ from $Z_2$.

\begin{proof}
  Let $\Lambda_0=\{ \lambda\in  \Lambda: D(\lambda;Z_1)\ne D(\lambda;Z_2)\}$. We first show that $\Lambda_0$ is finite. Let $\gamma$ be any curve in $D$ connecting $Z_1$ with $Z_2$. Since $\gamma\cap\rho$ is a compact subset of $\bigcup_{\lambda\in\Lambda} \lambda$, and every $\lambda\in\Lambda$ is a relatively open subset of $\rho$, we see that $\gamma$ intersects finitely many $\lambda\in\Lambda$. From the definition of $\Lambda_0$, $\gamma$ intersects every $\lambda\in\Lambda_0$.  Thus, $\Lambda_0$ is finite. We emphasize here that the above argument does not exclude the possibility that $\Lambda_0$ is empty.

  Next, we show that $\Lambda_0$ is nonempty. We choose $\gamma$ such that it minimizes the size of the set  $\Lambda(\gamma):=\{\lambda\in\Lambda:\gamma\cap\lambda\ne\emptyset\}$, which can not be empty since $\bigcup_{\lambda\in\Lambda} \lambda=\rho\cap D$ disconnects $Z_1$ from $Z_2$ in $D$. Let $\lambda_0\in \Lambda(\gamma)$. Let $w_1$ and $w_2$ be the first point and the last point on $\gamma$, which lies on $\lambda_0$, respectively.  Let $\lambda_0'$ be the sub curve of $\lambda_0$ with end points $w_1$ and $w_2$. There is $\eps>0$ such that $\dist(\lambda_0',\lambda)>\eps$ for any $\lambda\in \Lambda\sem\{\lambda_0\}$. Suppose  $\lambda_0\not\in\Lambda_0$. Then $D(\lambda;Z_1)= D(\lambda;Z_2)$. We may choose for $j=1,2$, $w_j'$ on the part of $\gamma$ between $Z_j$ and $w_j$, which is very close to $w_j$, such that there is a curve $\gamma_\eps$ connecting $w_1'$ and $w_2'$ in $D(\lambda_0;Z_1)$, which stays in the $\eps$-neighborhood of $\lambda_0'$. Construct a new curve $\gamma'$ in $D$ connecting $Z_1$ and $Z_2$ by modifying $\gamma$ such that the part of $\gamma$ between $w_1'$ and $w_2'$ is replaced by $\gamma_\eps$. Then we find that $\Lambda(\gamma')=\Lambda(\gamma)\sem\{\lambda_0\}$, which contradicts the assumption on $\gamma$. Thus, $\Lambda_0\supset \Lambda(\gamma)$ is nonempty.

Finally, we need to show that there is $\lambda_1\in\Lambda_0$, which minimizes $\{D(\lambda;Z_1):\lambda\in\Lambda_0\}$ and maximizes $\{D(\lambda;Z_2):\lambda\in\Lambda_0\}$. This follows from the finiteness and nonemptyness of $\Lambda_0$ and the facts that for any $\lambda_1,\lambda_2\in\Lambda_0$,  one of $D(\lambda_1;Z_1)$ and $D(\lambda_2;Z_1)$ is a subset of the other, and the inclusion relation is reversed if $Z_1$ is replaced by $Z_2$.
\end{proof}

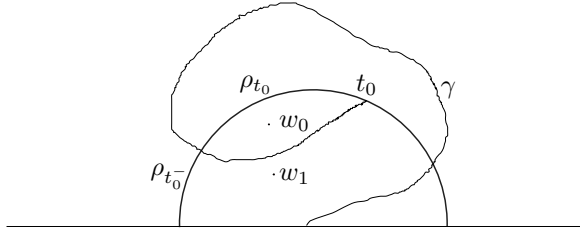
\begin{figure}
	
\labellist
\small
\pinlabel $\gamma$ at 460 160
\pinlabel $w_1$ at 300 70
\pinlabel $w_0$ at 300 125
\pinlabel $t_0$ at 375 165
\pinlabel $\rho_{t_0}$ at 260 165
\pinlabel $\rho_{t_0^-}$ at 170 70
\endlabellist	
		\centering
	\includegraphics[width=3in]{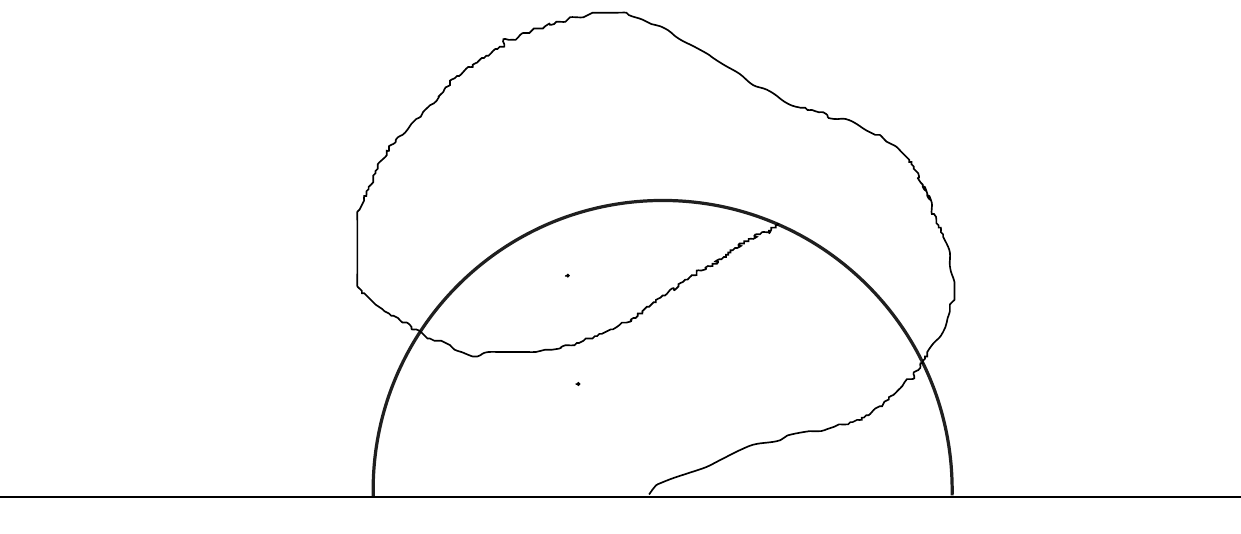}
	\caption{This figure illustrates the situation of Lemma \ref{right-ctns-B}. Here $D$ is the upper half plane, $\rho$ is the semi-circle, and $w_\infty$ is the prime end $\infty$. The curve $\gamma$ is shown up to some time $t_0$. Since $\rho_{t_0}$ does not disconnect $w_1$ from $\infty$, we have $f(t_0)=1$. When $t<t_0$ and is close to $t_0$, $\rho_t$ is the $\rho_{t_0^-}$ in the figure, which disconnects $w_1$ from $\infty$. This means that $f(t_0^-)=0$, and $f$ is not left-continuous at $t_0$. }
\end{figure}

\begin{Lemma}
  Let $D$ be a simply connected domain and $\rho$ a crosscut in $D$.  Let $w_0$, $w_1$ and $w_\infty$ be connected subsets or prime ends of $D$ such that $D\sem \rho$ is a neighborhood of all of them in $D$. Suppose that $\rho$ disconnects $w_0$ from $w_\infty$ in $D$. Let $\gamma(t)$, $0\le t<T$, be a continuous curve in $\lin D$ with $\gamma(0)\in\pa D$. Suppose for $0\le t<T$, $D\sem\gamma[0,t]$ is a neighborhood of $w_0$, $w_1$ and $w_\infty$ in $D$, and $w_0, w_1\subset D_t:=D(\gamma[0,t];w_\infty)$. For $0\le t<T$, let $\rho_t$ be the first sub-crosscut of $\rho$ in $D_t$ that disconnects $w_0$ from $w_\infty$ as given by Lemma \ref{first-crosscut}. For $0\le t<T$, let $f(t)=1$ if $w_1\in D_t(\rho_t;w_\infty)$; $=0$ if $w_1\in D_t^*(\rho_t;w_\infty)$. Then $f$ is right-continuous on $[0,T)$, and left-continuous at those $t_0\in(0,T)$ such that $\gamma(t_0)$ is not an end point of $\rho_{t_0}$.
\end{Lemma} \label{right-ctns-B}
\no{\bf Remark.} It is easy to see that $(D_t)_{0\le t<T}$ is a decreasing family of $\HH$-domains. But $(\rho_t)_{0\le t<T}$ may not be a decreasing family.

\begin{proof}
  We first show that $f$ is right-continuous. Fix $t_0\in[0,T)$. From the definition of $\rho_{t_0}$, there exist a curve $\beta_0$ in $D_{t_0}$, which goes from $w_0$ to $w_\infty$, crosses $\rho_{t_0}$ for only once, and does not visit $\rho\sem\rho_{t_0}$ before $\rho_{t_0}$. Let $S=w_\infty$ or $w_0$ depending on whether $f(t_0)=1$ or $0$. Then there is a curve $\beta_1$ in $D_{t_0}\sem \rho_{t_0}$ that connects $w_1$ with $S$. Since $\gamma(t_0)\not\in D_{t_0}$ and $\gamma$ is continuous, there is $t_1\in(t_0,T)$ such that $\gamma[t_0,t_1)$ is disjoint from $\beta_0$ and $\beta_1$. Fix $t\in(t_0,t_1)$. Then $\beta_0,\beta_1\subset D_t$. From Lemma \ref{first-crosscut}, there is the first sub-crosscut of $\rho_{t_0}$, denoted by $\rho_{t_0,t}$ in $D_t$ that disconnects $w_0$ from $w_\infty$. From the properties of $\beta_0$, $\rho_{t_0,t}$ is the connected component of $\rho_{t_0}\cap D_t$ that contains $\beta_0\cap\rho_{t_0}$. Since $\beta_0$ does not intersect $\rho$ before $\beta_0\cap\rho_{t_0}$, we have $\rho_t=\rho_{t_0,t}\subset \rho_{t_0}$. Thus, $\beta_1$ is a curve in $D_t\sem\rho_t$ connecting $w_1$ with $S$, which implies that $f$ is constant on $[t_0,t_1)$.

   Suppose $\gamma(t_0)$ is not an end point of $\rho_{t_0}$ for some $t_0\in (0,T)$. We now show that $f$ is left-continuous at $t_0$. There exists $t_1\in[0,t_0)$ such that $\gamma(t_1,t_0]$ does not intersect $\rho_{t_0}$. Fix $t\in(t_1,t_0]$. Then $\rho_{t_0}$ is a crosscut in $D_t$. Let $\beta_0,S,\beta_1$ be as above. Then $\beta_0$ and $\beta_1$ are also curves in $D_t$. From the properties of $\beta_0$, we see that $\rho_t=\rho_{t_0}$. Thus, $\beta_1$ is a curve in $D_t\sem\rho_t$ connecting $w_1$ with $S$, which implies that $f$ is constant on $(t_1,t_0]$.
\end{proof}

\subsection{Estimates}
We give some important estimates for SLE in this subsection. The first one is the interior estimate. To begin with, we quote the following theorem proved in \cite{Bf}.

\begin{Theorem} \label{oneptthm}
Suppose $\gamma$ is an $SLE_\kappa$ curve from $w_1$ to $w_2$ in a simply connected domain $D$. If $z \in D$, then
\[
\PP[ \dist (\gamma,z) \leq r]\leq C G_{(D;w_1,w_2)}(z) r^{2-d},
\]
where $G_{(D;w_1,w_2)}$ is the $1$-point Green's function for the $\gamma$.
\end{Theorem}

A stronger estimate is obtained in \cite{LR}: $\PP[ \dist (\gamma,z)\le r]=r^{2-d}G_{(D;w_1,w_2)}(z)[1+o(r^\alpha)]$, $\alpha>0$.
Using (\ref{green}), (\ref{green-D}) and Koebe's $1/4$ theorem, we find that $G_{(D;w_1,w_2)}(z)\le C \dist(z,\pa D)^{d-2}$. So we have the following interior estimate which is a corollary of Theorem \ref{oneptthm}.

\begin{Lemma} {\bf [Interior estimate]}
  For any $z\in D$,
  $$\PP[ \dist (\gamma,z) \leq r] \le C\Big(\frac{r}{\dist(z,\pa D)}\Big)^{2-d}.$$
\end{Lemma}


We will state the boundary estimate for SLE in several different forms. The original one comes from   \cite{Albert-Kozdron}, which is the following theorem.

\begin{Theorem} {\bf [Boundary estimate v.0]}
Let $\gamma$ be an $SLE_\kappa$ curve in $\HH$ from $0$ to $\infty$. Then for any $x_0\in\R\sem\{0\}$ and $r>0$,
$$\PP[ \dist (\gamma,x_0) \leq r] \le C\Big(\frac{r}{|x_0|}\Big)^{\alpha}.$$ \label{A-K}
\end{Theorem}


We will express the above theorem in another form using the notation of extremal distance. The reader may refer to \cite{Pom-bond} for the definition and properties of extremal distance (length). We use $d_D(L_1,L_2)$ to denote the extremal distance between $L_1$ and $L_2$ in $D$. Suppose $K$ is a nonempty $\HH$-hull with $\lin K\cap \lin\R_-=\emptyset$. Let $x_K=\max\{\lin K\cap\R\}$ and $r_K=\max\{|z-x_K|:z\in \lin K\}$. It is well known that there are absolute constants $C$ and $M$ such that $\frac{r_K}{x_K}\le C e^{-\pi d_{\HH}(K,\R_-)}$ if $d_{\HH}(K,\R_-)\ge M$. So the above theorem implies the following corollary.

\begin{Lemma} {\bf [Boundary estimate v.1]}
  Let $\gamma$ be as above. Then for any $\HH$-hull $K$ with $\lin K\cap \lin\R_-=\emptyset$, we have
$$\PP[ \gamma\cap K\ne \emptyset] \le Ce^{ -\alpha\pi d_{\HH}(K,\R_-)}.$$
The same is true if $\R_-$ is replaced with $\R_+$. \label{boundary-lem1}
\end{Lemma}

Using conformal invariance and comparison principle of extremal distance, we immediately get the following version of boundary estimate from the previous one.

\begin{Lemma} {\bf [Boundary estimate v.2]}
Let $D$ be a regular simply connected domain, and $w_0$ and $w_\infty$ be two distinct prime ends of $D$.  Let $\rho$ and $\eta$ be two disjoint crosscuts in $D$ such that $D(\rho;\eta)$ is neither a neighborhood of $w_0$ nor a neighborhood of $w_\infty$ in $D$. For $w_0$, the condition means that either $D\sem \rho$ is a neighborhood of $w_0$ and $D(\rho;w_0)=D^*(\rho;\eta)$, or $w_0$ is a prime end determined by $\rho$; and likewise for $w_\infty$. Let $\gamma$ be an $SLE_\kappa$ curve in $D$ from $w_0$ to $w_\infty$. Then
$$ \PP[\gamma\cap  ({\eta}\cup D^*(\eta;w_\infty))\ne\emptyset]\le C e^{-\alpha \pi d_D(\rho,\eta)}. $$
\end{Lemma}


We now combine the interior estimate and the boundary estimate to get the following one-point estimate, which implies the case $n=1$ in Theorem \ref{mainthm2}.

\begin{Lemma} {\bf [One-point estimate]}
Let $D$ be an $\HH$-domain with a prime end $w_0\ne\infty$. Let $\gamma$ be an $SLE_\kappa$ curve in $D$ from $w_0$ to $\infty$.
  Let $z_0\in\lin\HH$, $y_0=\Imm z_0\ge0$, and $R>r>0$. Let $\rho=\{z\in\HH:|z-z_0|=R\}$ and $\eta=\{z\in\HH:|z-z_0|=r\}$. Suppose $\{z\in\HH:|z-z_0|\le R\}\subset D$ and $w_0\not\in \{x\in\R:|x-z_0|<R\}$. Then
  $$\PP[\gamma\cap \eta\ne\emptyset]\le C\frac{P_{y_0}(r)}{P_{y_0}(R)}.$$ \label{one-point-lemma}
\end{Lemma}
\begin{proof}
  We consider different cases.
  Case 1: $y_0 \ge R$. The conclusion follows from the interior estimate because $\frac{P_{y_0}(r)}{P_{y_0}(R)}=(\frac r R)^{2-d}$ and $\dist(z_0,\pa D)\ge R$. Case 2: $y_0\le r$. We have $\frac{P_{y_0}(r)}{P_{y_0}(R)}=(\frac r R)^{\alpha}$. By increasing the value of $C$, we may assume that $R> 4r$. The conclusion follows from the boundary estimate because $\rho$ and $\eta$ are separated in $D$ by the two crosscuts $\{z\in\HH: |z-\Ree z_0|=2r\}$ and $\{z\in\HH: |z-\Ree z_0|=R/2\}$, and the extremal distance between them in $D$ is $\log(R/(4r))/\pi$. Case 3: $R> y_0> r$. Let $\rho'=\{z\in \HH:|z-z_0|=y_0\}$, which separates $\rho$ from $\eta$ in $D$. Let $T$ be the first time that $\gamma$ hits $\rho'$, and $\gamma^T(t)=\gamma(T+t)$, $0\le t<\infty$, if $T<\infty$. Then $T$ is an stopping time, and $\{\gamma\cap \eta\ne\emptyset\}=\{\gamma^T\cap\eta\ne\emptyset\}\subset \{T<\infty\}$ almost surely. From the result of Case 2, $\PP[T<\infty]\le C\frac{P_{y_0}(y_0)}{P_{y_0}(R)}$. From DMP, conditioned on $\gamma[0,T]$ and $\{T<\infty\}$, the $\gamma^T$ is an $SLE_\kappa$ curve in $D(\gamma[0,T])$ from $\gamma(T)$ to $\infty$. Since $\dist(z_0,\pa D_T)=y_0$, from the result of Case 1, we get
  $\PP[\gamma^T\cap \eta\ne\emptyset|\gamma[0,T],T<\infty]\le C \frac{P_{y_0}(r)}{P_{y_0}(y_0)}$.  Combining this with the estimate for $\PP[T<\infty]$, we get the conclusion in Case 3.
\end{proof}

The following version of boundary estimate will be frequently used in this paper.

\begin{Lemma} {\bf [Boundary estimate v.3]}
Let $D$ be an $\HH$-domain with a prime end $w_0\ne\infty$. Let $\gamma$ be an $SLE_\kappa$ curve in $D$ from $w_0$ to $\infty$.
Let $\rho$ be a crosscut in $D$ such that $D^*(\rho)$ is not a neighborhood of $w_0$ in $D$, and $S\subset D^*(\rho)$. Let $\til D$ be a domain that contains $D$, and $\til\rho$ a subset of $\til D$ that contains $\rho$. Let $\til\eta$ be a Jordan curve in $\til D$, which intersects $\pa D$, or a crosscut in $\til D$. Suppose that $\til {\eta}$ disconnects $S$ from $\til \rho$ in $\til D$. Then
  $$\PP[\gamma\cap S\ne\emptyset]\le C e^{-\pi \alpha  d_{\til D}(\til\rho,\til{\eta})}. $$\label{boundary-lem}
\end{Lemma}
\begin{proof}
From Lemma \ref{first-crosscut}, $\til{\eta}$ contains a sub-crosscut in $E$, denoted by $\eta$, which disconnects $S$ from $\rho$. Since $S\subset D^*(\rho)$, we have $\eta\subset D^*(\rho)$ and $S\subset D^*(\eta)$. Thus, $D(\rho;\eta)=D^*(\rho)$ is not a neighborhood of either $\infty$ or $w_0$ in $D$.  Using the boundary estimate v.2, we get
$$\PP[\gamma\cap S\ne\emptyset]\le \PP[\gamma\cap D^*(\eta)\ne\emptyset]\le C e^{-\pi \alpha  d_D(\rho,\eta)}\le C e^{-\pi \alpha  d_{\til D}(\til \rho,\til {\eta})}.$$
\end{proof}





\section{Main Lemmas}  \label{Mainsec}
In this section, we let $\gamma$ be an $SLE_\kappa$ curve in $\HH$ from $0$ to $\infty$. Given any set $S$, let $\tau_S=\inf\{t\ge 0:\gamma(t)\in S\}$; we set $\inf\emptyset=\infty$ by convention.
Let $(\F_t)$ be the right-continuous filtration generated by $\gamma$. For $t_0\ge 0$, let $\gamma^{t_0}(t)=\gamma(t_0+t)$, $0\le t<\infty$, and $H_{t_0}=\HH(\gamma[0,t_0])$. Recall the DMP: if $T$ is an $(\F_t)$-stopping time, then conditioned on $\F_T$ and $T<\infty$, $\gamma^T$ is an $SLE_\kappa$ curve in $H_T$ from (the prime end of $H_T$ determined by) $\gamma(T)$ to $\infty$.

\begin{Theorem}
Let $m\in\N$, $z_j\in\lin\HH$ and $ R_j\ge r_j>0$, $0\le j\le m$. Let $\ha\xi_j=\{|z-z_j|=R_j\}$, $\xi_j=\{|z-z_j|=r_j\}$, and $\ha D_j=\{|z-z_j|\le R_j\}$, $0\le j\le m$. Suppose that $0\not\in \ha D_j$, $0\le j\le m$; and $\ha D_0\cap \ha D_j=\emptyset$, $1\le j\le m$. Let $r_0'\in (0,r_0)$ and $\xi_0'=\{|z-z_0|=r_0'\}$.
Let $$E=\{\tau_{\xi_0}<\tau_{\ha \xi_1}\le \tau_{\xi_1}< \cdots< \tau_{\ha\xi_m}\le \tau_{\xi_m}<\tau_{\xi_0'}<\infty\}.$$
Let $y_j=\Imm z_j$, $1\le j\le m$. Then we have
  $$\PP[{E}|\F_{\tau_{\xi_0}}]\le C^m \Big(\frac{r_0}{R_0}\Big)^{\alpha/4} \prod_{j=1}^m  \frac{P_{y_j}(r_j)}{P_{y_j}(R_j)}.$$
  \label{key-lem}
\end{Theorem}

\no{\bf Discussion.} From the $1$-point estimate, we see that, given $\gamma$ up to hitting $\ha\xi_j$, the probability that it reaches $\xi_j$ is at most $C \frac{P_{y_j}(r_j)}{P_{y_j}(R_j)}$. The DMP allows us to put these estimates together to get the product on the righthand side of the above formula. The key point of the proof is to use the boundary estimate to derive the factor $(\frac{r_0}{R_0})^{\alpha/4}$. Recall that the boundary estimate can be applied when the $SLE$ curve is required to cross a disjoint pair of crosscuts from the unbounded component to the bounded component determined by these crosscuts. But whether a given set lies in the bounded component may vary as the $SLE$ curve grows. So we have to carefully keep track of the changes of the ``topology'' situations.

\begin{proof} Let $\Xi$ be the set of $\xi_j$, $\ha\xi_j$, $0\le j\le n$, and $\xi_0'$. By Theorem \ref{A-K}, for any $\xi\in\Xi$, $\gamma$ almost surely does not visit $\xi\cap\R$. By discarding an event with probability zero, we may assume that $\gamma$ does not visit $\xi\cap \R$ for any $\xi\in\Xi$. Then for any $\xi\in\Xi$, $\tau_\xi=\tau_{\xi\cap\HH}$. Thus, it suffices to prove the lemma with each $\xi\in\Xi$ replaced by $\xi\cap\HH$. This means that every $\xi\in\Xi$ is a Jordan curve or crosscut in $\HH$. After that, we see that $\tau_\xi<\infty$ implies that $\gamma(\tau_\xi)\in\xi\cap\HH$, and $\gamma$ does not visit $\HH^*(\xi)$ before $\xi$.

Let $\tau_0=\tau_{\xi_0}$, $\ha\tau_j=\tau_{\ha\xi_j}$ and $\tau_j=\tau_{\xi_j}$, $1\le j\le m$, and $\tau_{m+1}=\tau_{\xi_0'}$.
From the DMP and one-point estimate (Lemma \ref{one-point-lemma}), we get
\BGE \PP[\tau_j<\infty|\F_{\ha\tau_{j}}]\le  C\frac{P_{y_j}(r_j)}{P_{y_j}(R_j)},\quad 1\le j\le m.\label{1-pt*}\EDE
Thus, $\PP[E|\F_{\tau_0}]\le C^m\prod_{j=1}^m \frac{P_{y_j}(r_j)}{P_{y_j}(R_j)}$. If $R_0=r_0$, the proof is finished.


Suppose $R_0>r_0$. Let $\rho=\{z\in\HH:|z-z_0|=\sqrt{R_0r_0}\}$. Then $\rho$ is a Jordan curve or crosscut in $\HH$, which lies between $\ha\xi_0$ and $\xi_0$, and
\BGE d_{\HH}(\rho,\xi_0),d_{\HH}(\rho,\ha\xi_0)\ge \frac{\log(R_0/r_0)}{4\pi}.\label{extremal}\EDE

Also note that $\rho$ disconnects $\xi_0'$ from $\infty$. Let $T=\inf\{t\ge 0:\xi_0'\not\subset H_t\}$. For $\tau_0\le t<T$, $\xi_0'$ is a connected subset of $H_t$, and $\rho$ intersects $\pa H_t$. Thus, we may use Lemma \ref{first-crosscut} to define $\rho_t$ to be the first sub-crosscut of $\rho$ in $H_t$ that disconnects $\xi_0'$ from $\infty$ for $\tau_0\le t<T$. Note that every $\rho_t$ is $\F_t$-measurable.

Let $I=\{(j,j+1):0\le j\le m\}\cup\{(j,j):1\le j\le m\}$, and define $(A_\iota)_{\iota\in I}$ by
$$A_{(0,1)}=\{T>\tau_0\}\cap\{\HH^*(\xi_1)\subset H_{\tau_0}^*(\rho_{\tau_0})\}\in\F_{\tau_0};$$
$$A_{(j,j)}=\{T>\tau_j\}\cap\{\HH^*(\xi_j)\subset H_{\tau_{j-1}}(\rho_{\tau_{j-1}})\}\cap \{\HH^*(\xi_j)\subset H_{\tau_j}^*(\rho_{\tau_j})\}\in\F_{\tau_j},\quad 1\le j\le m;$$
$$A_{(j,j+1)}=\{T>\tau_j\}\cap\{\HH^*(\xi_j)\subset H_{\tau_j}(\rho_{\tau_j})\}\cap \{\HH^*(\xi_{j+1})\subset H_{\tau_j}^*(\rho_{\tau_j})\}\in\F_{\tau_{j}},\quad 1\le j\le m-1;$$
$$A_{(m,m+1)}=\{T>\tau_m\}\cap\{\HH^*(\xi_m)\subset H_{\tau_m}(\rho_{\tau_m})\}\in\F_{\tau_m}.$$
Suppose $E$ occurs. Then $\gamma$ does not visit $\xi_0'$ at any time $t\le\tau_m$. So $\xi_0'$ is a connected subset of $\HH\sem \gamma[0,\tau_m]$. Then we must have $\xi_0'\subset H_{\tau_m}$ because $\gamma^{\tau_m}$ visits $\xi_0'$, and $\gamma^{\tau_m}\subset \lin{H_{\tau_m}}\subset H_{\tau_m}\cup \gamma[0,\tau_m]$. Thus, $T>\tau_m>\tau_{m-1}>\cdots>\tau_1>\tau_0$. Similarly, since $\HH^*(\xi_j)$ is not visited by $\gamma$ at any time $t\le\tau_j$, we conclude that $\HH^*(\xi_j)\subset H_t$ for $t\le \tau_j$. Since $\HH^*(\xi_j)$ is disjoint from $\rho\supset \rho_t$, we conclude that $\HH^*(\xi_j)$ is contained in either $H_t(\rho_t)$ or $H_t^*(\rho_t)$ for any $t\le \tau_j$.

Define a strict total order on $I$ such that $(0,1)<(1,1)<(1,2)<(2,2)<\cdots<(m-1,m)<(m,m)<(m,m+1)$. Define a family of events $E_\iota$, $\iota\in I$, such that $E_\iota=E\sem\bigcup_{\iota':\iota'>\iota} A_{\iota'}$. Using induction, one can prove that
    $$E_{\iota}\subset \{\HH^*(\xi_{\iota_1})\subset H_{\tau_{\iota_2}}^*(\rho_{\tau_{\iota_2}})\},\quad \iota=(\iota_1,\iota_2)\in I\sem \{(m,m+1)\}.$$
    Especially, we get $$E_{0,1}=E\sem\bigcup_{\iota\in I\sem\{(0,1)\}} A_\iota\subset \{\HH^*(\xi_0)\subset H_{\tau_1}^*(\rho_{\tau_1})\}\subset A_{(0,1)}.$$
Thus, we have $E\subset\bigcup_{\iota\in I} A_\iota$. We will finish the proof by showing that
\BGE \PP[E\cap A_{\iota}|\F_{\tau_0}]  \le C^m \Big(\frac{r_0}{R_0}\Big)^{\alpha/4} \prod_{j=1}^m  \frac{P_{y_j}(r_j)}{P_{y_j}(R_j)}, \quad \iota\in I.\label{conclusion*}\EDE

\no{\bf Case 1.} Suppose $A_{(0,1)}$ occurs and $\tau_0<\ha\tau_1$. Since $\ha\xi_1$ and $\HH^*(\xi_1)$ are subsets of $\HH^*(\ha\xi_1)\cup\ha\xi_1$, which is a connected subset of $(\HH\sem \gamma[0,\tau_0])\sem \rho_{\tau_0}$, from $\HH^*(\xi_1)\subset H_{\tau_0}^*(\rho_{\tau_0})$, we conclude that $\ha\xi_1\subset H_{\tau_0}^*(\rho_{\tau_0})$. Note that $\rho$ disconnects $\ha\xi_1$ from $\xi_0'$  in $\HH$, and intersects $\pa H_{\tau_0}$. Applying Lemma \ref{first-crosscut}, we get a sub-crosscut of $\rho$, denoted by $\rho_{\tau_0}'$, that disconnects $\ha\xi_1$ from $\xi_0'$ in $H_{\tau_0}$. Since both $\ha\xi_1$ and $\xi_0'$ lie in $H^*_{\tau_0}(\rho_{\tau_0})$, so does $\rho_{\tau_0}'$. Thus, $H^*_{\tau_0}(\rho_{\tau_0}')\subset H^*_{\tau_0}(\rho_{\tau_0})$. Since $\rho_{\tau_0}$ is the first sub-crosscut of $\rho$ in $H_{\tau_0}$ that disconnects $\xi_0'$ from $\infty$, we see that $\rho_{\tau_0}'$ does not disconnect $\xi_0'$ from $\infty$. Thus, $\xi_0'\subset H_{\tau_0}(\rho_{\tau_0}')$, and $\ha\xi_1\subset H_{\tau_0}^*(\rho_{\tau_0}')$ as $\rho_{\tau_0}'$ disconnects $\ha\xi_1$ from $\xi_0'$ in $H_{\tau_0}$. See Figure 1.

\begin{figure}
\labellist
\small
\pinlabel $\rho'_{\tau_0}$ at 300 205
\pinlabel $\xi_1$ at 150 265
\pinlabel $\ha{\xi}_1$ at 120 290
\pinlabel $\xi'_0$ at 405 193
\pinlabel $\xi_0$ at 405 228
\pinlabel $\rho$ at 405 270
\pinlabel $\rho_{\tau_0}$ at 450 65
\pinlabel $\ha{\xi_0}$ at 405 320
\pinlabel $0$ at 320 -15
\pinlabel $\gamma$ at 230 100
\endlabellist
\centering
\includegraphics[width=3in]{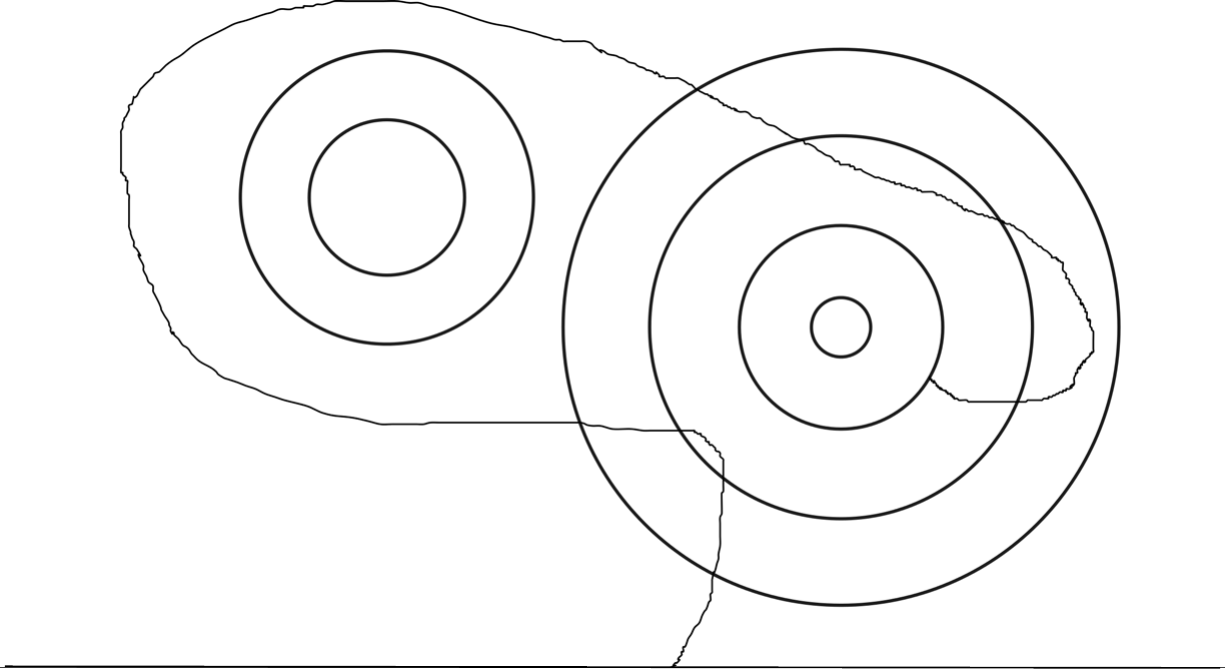}
\caption{This figure shows the event $A_{(0,1)}$ with $\gamma$ stopped at $\tau_0=\tau_{\xi_0}$.} \label{Fig2}
\end{figure}
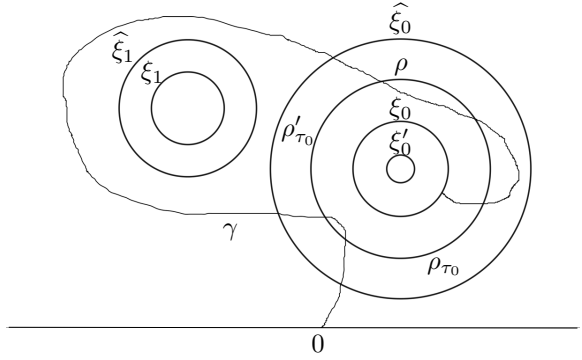

Since $\HH^*(\xi_0)$ is a connected subset of $H_{\tau_0}\sem\rho_{\tau_0}'$, and contains $\xi_0'$ and a curve that approaches $\gamma(\tau_0)\in\xi_0$, we conclude that $H_{\tau_0}(\rho_{\tau_0}'; \gamma(\tau_0))=H_{\tau_0}(\rho_{\tau_0}'; \xi_0')=H_{\tau_0}(\rho_{\tau_0}')$. Thus, $H_{\tau_0}(\rho_{\tau_0}';\ha\xi_1) =H_{\tau_0}^*(\rho_{\tau_0}')$ is not a neighborhood of $\gamma^{\tau_0}(0)=\gamma(\tau_0)$ in $H_{\tau_0}$.
Since $\tau_0<\ha\tau_1$, $\ha\tau_1<\infty$ implies that the $SLE_\kappa$ curve $\gamma^{\tau_0}$ in $H_{\tau_0}$ (conditioned on $\F_{\tau_0}$) visits $\ha\xi_1$. Since $\ha\xi_0$ disconnects $\ha\xi_1$ from $\rho\supset \rho_{\tau_0}'$ in $\HH$, and intersects $\pa H_{\tau_0}$, from the boundary estimate v.3 (Lemma \ref{boundary-lem}) and (\ref{extremal}), we get
$$\PP[\ha\tau_1<\infty|\F_{\tau_0},A_{(0,1)},\tau_0<\ha\tau_1]\ \le C e^{-\alpha \pi d_{\HH}(\rho, \ha\xi_0)}\le   C \Big(\frac{r_0}{R_0}\Big)^{\alpha/4},$$
which together with (\ref{1-pt*}) implies that (\ref{conclusion*}) holds for $\iota=(0,1)$.

\vskip 3mm

\no{\bf Case 2.} Suppose for some $1\le j\le m-1$, $A_{(j,j+1)}$ occurs and $\tau_j<\ha\tau_{j+1}$. Using the argument in the previous case with $\tau_0$ and $\ha\xi_1$ replaced by $\tau_j$ and $\ha\xi_{j+1}$, respectively, we get  a sub-crosscut of $\rho$, denoted by $\rho_{\tau_j}'$, that disconnects $\ha\xi_{j+1}$ from $\xi_0'$ in $H_{\tau_j}$, and conclude that $H^*_{\tau_j}(\rho_{\tau_j}')\subset H^*_{\tau_j}(\rho_{\tau_j})$, $\xi_0'\subset H_{\tau_j}(\rho_{\tau_j}')$, and $\ha\xi_{j+1}\subset H_{\tau_j}^*(\rho_{\tau_j}')$.

Since $\HH^*(\xi_j) $ is a connected subset of $H_{\tau_j}\sem \rho_{\tau_j}$, and contains a curve that approaches $\gamma(\tau_j)\in \xi_j$, we conclude that $H_{\tau_j}(\rho_{\tau_j}; \gamma(\tau_j))=H_{\tau_j}(\rho_{\tau_j};\HH^*(\xi_j))=H_{\tau_j}(\rho_{\tau_j})$. Thus, $H_{\tau_j}^*(\rho_{\tau_j}')\subset H_{\tau_j}^*(\rho_{\tau_j})$ is not a neighborhood of $\gamma^{\tau_j}(0)=\gamma(\tau_j)$ in $H_{\tau_j}$.
Since $\tau_j<\ha\tau_{j+1}$, $\ha\tau_{j+1}<\infty$ implies that the $SLE_\kappa$ curve $\gamma^{\tau_j}$ in $H_{\tau_j}$ (conditioned on $\F_{\tau_j}$) visits $\ha\xi_{j+1}$. Since $\ha\xi_0$ disconnects $\ha\xi_{j+1}$ from $\rho\supset\rho_{\tau_j}'$ in $\HH$,  from Lemma \ref{boundary-lem} and (\ref{extremal}), we get
$$\PP[\ha\tau_{j+1}<\infty|\F_{\tau_j},A_{(j,j+1)},\tau_j<\ha\tau_{j+1}]\le  C e^{-\alpha \pi d_{\HH}(\rho, \ha\xi_0)}\le C\Big(\frac{r_0}{R_0}\Big)^{\alpha/4},$$
which together with (\ref{1-pt*}) implies that (\ref{conclusion*}) holds for $\iota=(j,j+1)$, $1\le j\le m-1$.

\vskip 3mm
\no{\bf Case 3.} Suppose $A_{(m,m+1)}$ and $\tau_m<\tau_{m+1}$ occur. Since $\HH^*(\xi_m) $ is a connected subset of $H_{\tau_m}\sem \rho_{\tau_m}$, and contains a curve that approaches $\gamma(\tau_m)\in\xi_m$, we conclude that $H_{\tau_m}(\rho_{\tau_m}; \gamma(\tau_m))=H_{\tau_m}(\rho_{\tau_m}; \HH^*(\xi_m))=H_{\tau_m}(\rho_{\tau_m})$. Thus, $H_{\tau_m}^*(\rho_{\tau_m})$ is not a neighborhood of $\gamma^{\tau_m}(0)=\gamma(\tau_m)$ in $H_{\tau_m}$. Since $\tau_m<\tau_{m+1}$, $\tau_{m+1}<\infty$ implies that the $SLE_\kappa$ curve $\gamma^{\tau_m}$ in $H_{\tau_m}$ (conditioned on $\F_{\tau_m}$) visits $\xi_0'\subset H_{\tau_m}^*(\rho_{\tau_m})$. Since $\xi_0$ disconnects $\xi_0'$ from $\rho$ in $\HH$, and intersects $\pa H_{\tau_m}$, we may apply Lemma \ref{boundary-lem} and (\ref{extremal}) to get
$$\PP[\tau_{m+1}<\infty|\F_{\tau_m},A_{(m,m+1)},\tau_m<\tau_{m+1}]\le C e^{-\pi d_{\HH}(\xi_0,\rho)}\le C\Big(\frac{r_0}{R_0}\Big)^{\alpha/4},$$
which together with (\ref{1-pt*}) implies that (\ref{conclusion*}) holds for $\iota=(m,m+1)$.

\vskip 3mm
\no{\bf Case 4.} Finally, we consider (\ref{conclusion*}) in the case $\iota=(j,j)$. Fix $1\le j\le m$ and define
$$\sigma_j=\inf\{t\ge \tau_{j-1}: \HH^*(\xi_j)\subset H_{t}^*(\rho_{t})\}.$$
From Lemma \ref{right-ctns-B} and the right-continuity of $(\F_t)$, we have
\begin{enumerate}
 \item Every $\sigma_j$ is an  $(\F_t)$-stopping time.
  \item If $\sigma_j<\infty$, then $\HH^*(\xi_j)\subset H_{\sigma_j}^*(\rho_{\sigma_j})$.
  \item If $A_{(j,j)}$ occurs, then $\tau_{j-1}< \sigma_j<\tau_j$.
  \item If $\tau_{j-1}< \sigma_j<\infty$, then $\gamma(\sigma_j)$ is an endpoint of ${\rho_{\sigma_j}}$.
\end{enumerate}
Note that the last property implies that $H_{\sigma_j}^*(\rho_{\sigma_j})$ is not a neighborhood of either $\gamma(\sigma_j)$ or $\infty$ in $H_{\sigma_j}$.
Let $F_<=\{\sigma_j<\ha\tau_j\}$ and $F_\ge=\{\ha\tau_j\le \sigma_j<\tau_j\}$. Then $A_{(j,j)}\subset F_<\cup F_\ge$.

\begin{figure}
	
		\labellist
		\small
		\pinlabel $\xi_0$ at 360 285
		\pinlabel $\rho$ at 420 335
		\pinlabel $\xi'_0$ at 400 255
		\pinlabel $\ha{\xi}_j$ at 260 500
		\pinlabel $\ha{\xi}_{j+1}$ at 195 290
		\pinlabel $\rho'_{\tau_j}$ at 285 220
		\pinlabel $\xi_j$ at 190 485
		\pinlabel $\xi_{j+1}$ at 120 255
		\pinlabel $\ha{\xi}_0$ at 440 380
		\pinlabel $0$ at 290 15
		\pinlabel $\rho_{\tau_j}$ at 400 85
		\pinlabel $\gamma$ at 380 460
		\endlabellist
	\centering
	\includegraphics[width=3in]{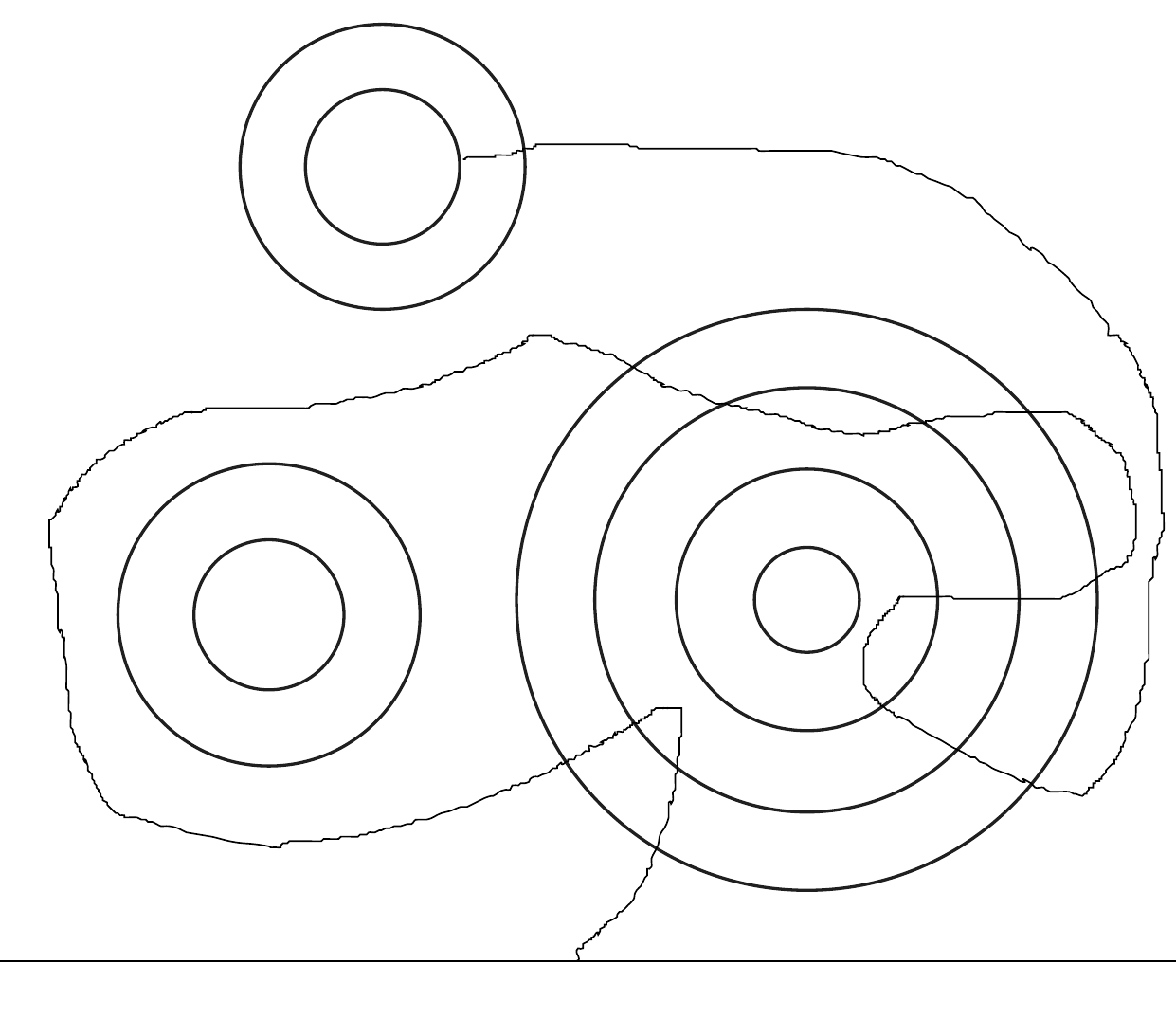}
	\caption{This figure shows the event $A_{(j,j+1)}$ with $\gamma$ stopped at $\tau_j=\tau_{\xi_j}$. }
\end{figure}
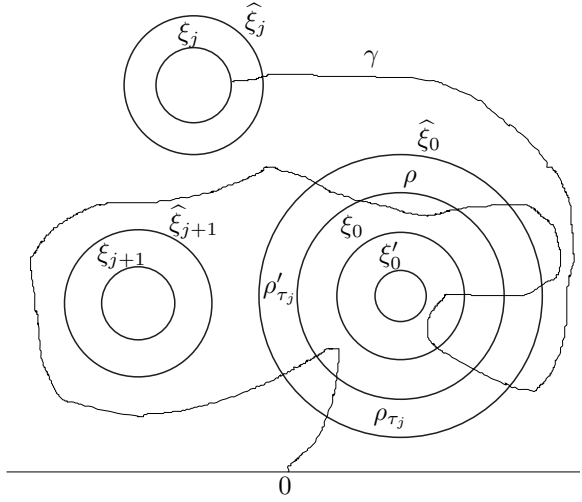

\vskip 3mm
\no{\bf Case 4.1.} Suppose $F_\ge$ occurs. Let $N=\lceil \log(R_j/r_j)\rceil\in\N$. Let $\zeta_k=\{z\in\HH:|z-z_j|=(R_j^{N-k} r_j^k)^{1/N}\}$, $0\le k\le N$. Note that $\zeta_0=\ha\xi_j$ and $\zeta_N=\xi_j$. Then $F_{\ge }\subset \bigcup_{k=1}^N F_k$, where
$$F_k:=\{\tau_{\zeta_{k-1}}\le \sigma_j<\tau_{\zeta_k}\},\quad 1\le k\le N.$$
If $F_k$ occurs, then $\zeta_k\subset H_{\sigma_j}^*(\rho_{\sigma_j})$ because $\HH^*(\zeta_k)\cup\zeta_k$ is a connected subset of $(\HH\sem \gamma[0,\sigma_j])\sem \rho$ that contains both $\zeta_k$ and $\HH^*(\xi_j)$, and $\HH^*(\xi_j)\subset H_{\sigma_j}^*(\rho_{\sigma_j})$. See Figure 2.

From Lemma \ref{boundary-lem} and (\ref{extremal}), we get
$$\PP[\tau_{\zeta_k}<\infty|\F_{\sigma_j},F_k]\le C e^{-\alpha\pi d_{\HH}(\rho,\zeta_{k-1})}\le  C e^{-\alpha\pi (d_{\HH}(\rho,\ha\xi_0)+d_{\HH}(\zeta_0,\zeta_{k-1}))}\le C \Big(\frac{r_0}{R_0}\Big)^{\alpha/4} \Big(\frac{r_j}{R_j}\Big)^{\frac\alpha2 \frac{k-1}N}. $$
From Lemma \ref{one-point-lemma}, we get
$$ \PP[F_k|\F_{\tau_{j-1}}, \tau_{j-1}<\ha\tau_j]\le C \frac{P_{y_j}((R_j^{N-k+1} r_j^{k-1})^{1/N})}{P_{y_j}(R_j)}.$$ 
$$\PP[\tau_j<\infty |\F_{\tau_{\zeta_k}},F_k]\le C \frac{P_{y_j}(r_j)}{P_{y_j}((R_j^{N-k} r_j^{k})^{1/N})}.$$
The above three displayed formulas together with (\ref{P-compare}) imply that
$$ \PP[\tau_j<\infty,F_k|\F_{\tau_{j-1}},\tau_{j-1}<\ha\tau_j]\le  C \Big(\frac{r_0}{R_0}\Big)^{\alpha/4} \Big(\frac{r_j}{R_j}\Big)^{\frac\alpha2 \frac{k-1}N}\Big(\frac{r_j}{R_j}\Big)^{-\alpha/N}\frac{P_{y_j}(r_j)}{P_{y_j}(R_j)} .$$ 
Since $F_{\ge }\subset \bigcup_{k=1}^N F_k$, by summing up the above inequality over $k$, we get
\BGE \PP[\tau_j<\infty,F_{\ge} |\F_{\tau_{j-1}},\tau_{j-1}<\ha\tau_j]\le  C \Big(\frac{r_0}{R_0}\Big)^{\alpha/4} \frac{P_{y_j}(r_j)}{P_{y_j}(R_j)}\left[\Big(\frac{r_j}{R_j}\Big)^{-\alpha/N} \frac{1-(\frac{r_j}{R_j})^{\alpha/2}}{1-(\frac{r_j}{R_j})^{\alpha/(2N)}}\right].
\label{Fk*}\EDE
By considering the cases $R_j/r_j\le e$ and $R_j/r_j>e$ separately, we see that the quantity inside the square bracket is bounded by the constant $\frac{e^\alpha}{1-e^{-\alpha/4}}$.
\vskip 3mm

\no{\bf Case 4.2.} Suppose $F_{<}$ occurs. 
Then $\HH^*(\ha\xi_j)\cup\ha\xi_j$ is a connected subset of $(\HH\sem \gamma[0,\sigma_j])\sem \rho$ that contains $\HH^*(\xi_j)$. So we get $\ha\xi_j\subset H_{\sigma_j}^*(\rho_{\sigma_j};\HH^*(\xi_j)= H_{\sigma_j}^*(\rho_{\sigma_j})$. Since $\ha\xi_0$ disconnects $\rho$ from $\ha\xi_j$ in $\HH$, applying Lemma \ref{boundary-lem} and (\ref{extremal}), we get
$$\PP[\ha\tau_j <\infty|\F_{\sigma_j},F_{<}]\le C e^{-\alpha \pi d_{\HH}(\rho, \ha\xi_0)}\le C\Big(\frac{r_0}{R_0}\Big)^{\alpha/4}, $$
which together with (\ref{1-pt*}) implies that
\BGE \PP[\tau_j<\infty, F_{<}|\F_{\tau_{j-1}}]\le C \Big(\frac{r_0}{R_0}\Big)^{\alpha/4} \frac{P_{y_j}(r_j)}{P_{y_j}(R_j)}.\label{F0*}\EDE

Combining (\ref{Fk*}) and (\ref{F0*}),  we get
$$\PP[\tau_j<\infty, A_{(j,j)}|\F_{\tau_{j-1}},\tau_{j-1}<\ha\tau_j]\le C \Big(\frac{r_0}{R_0}\Big)^{\alpha/4} \frac{P_{y_j}(r_j)}{P_{y_j}(R_j)},$$
which together with (\ref{1-pt*}) implies that (\ref{conclusion*}) holds for $\iota=(j,j)$, $1\le j\le m$.
\end{proof}

\begin{figure}
	\labellist
	\small
	\pinlabel $\xi_0$ at 355 230
	\pinlabel $\rho$ at 355 280
	\pinlabel $\ha{\xi}_0$ at 355 350
	\pinlabel $\xi_{j-1}$ at 200 535
	\pinlabel $\ha{\xi}_{j-1}$ at 200 610
	\pinlabel $\zeta_k$ at 515 540
	\pinlabel $\xi_j$ at 535 500
	\pinlabel $\zeta_{k-1}$ at 510 600
	\pinlabel $\ha{\xi}_j$ at 470 650
	\pinlabel $0$ at 375 -15
	\pinlabel $\rho_{\sigma_j}$ at 460 150
	\pinlabel $\gamma$ at 200 300
	\endlabellist
	\centering
	\includegraphics[width=3.5in]{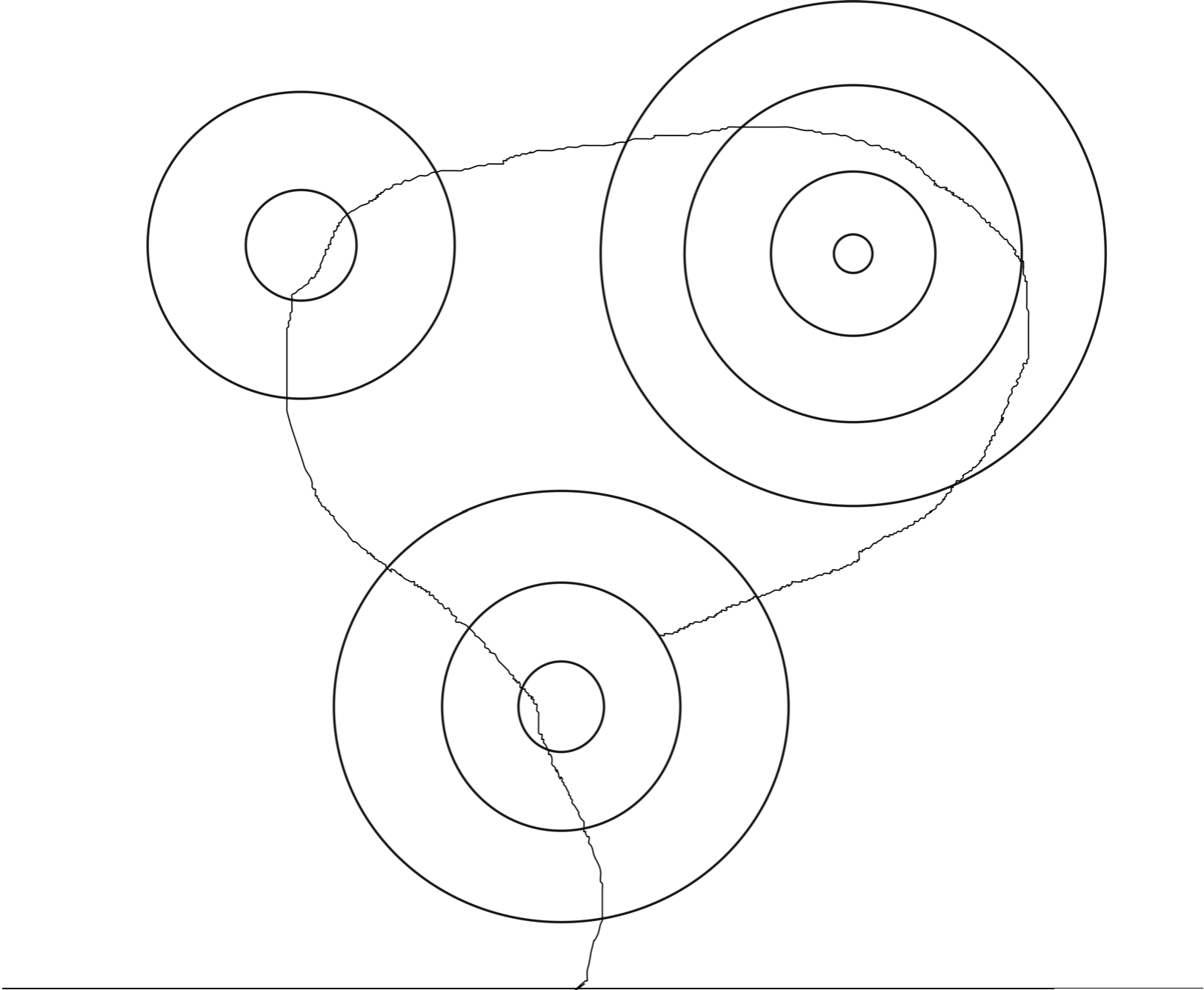}
	\caption{This figure shows the event $F_k$, a sub event of $A_{(j,j)}$,  with $\gamma$ stopped at $\sigma_j$, the first time after $\tau_{j-1}=\tau_{\xi_{j-1}}$ that $\xi_j$ lies in the bounded component of $H_t\sem \rho_t$.}  \label{fig1}
\end{figure}
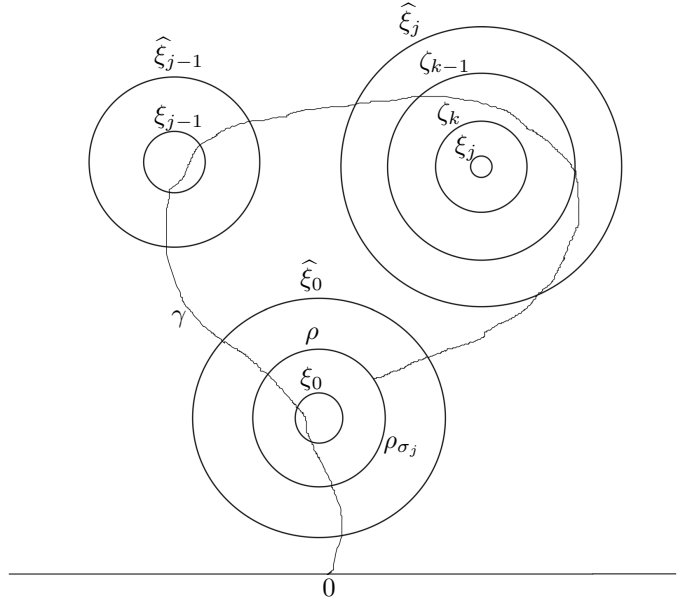

Let $\Xi$ be a family of mutually disjoint circles with center in $\lin\HH$, each of which does not pass through or enclose $0$. Define a partial order on $\Xi$ such that $\xi_1<\xi_2$ if $\xi_2$ is enclosed by $\xi_1$. One should keep in mind that a smaller element in $\Xi$ has bigger radius, but will be visited earlier (if it happens) by a curve started from $0$.

Suppose that $\Xi$ has a partition $\{\Xi_e\}_{e\in\cal E}$ with the following properties:
\begin{enumerate}
  \item For each $e\in\cal E$, the elements in $\Xi_e$ are concentric circles with radii forming a geometric sequence with common ratio $1/4$.
  We denote the common center $z_e$, the biggest radius $R_e$, and the smallest radius $r_e$.
  \item Let $A_e=\{r_e\le |z-z_0|\le R_e\}$ be the closed annulus associated with $\Xi_e$, which is a single circle if $R_e=r_e$, i.e., $|\Xi_e|=1$. Then the annuli $A_e$, $e\in\cal E$, are mutually disjoint.
\end{enumerate}
Note that every $\Xi_e$ is a totally ordered set w.r.t.\ the partial order on $\Xi$.

\begin{Theorem}
  Let $y_e:=\Imm z_e\ge 0$, $e\in\cal E$. Then there is $C_{|\cal E|}<\infty$, which depends only on $\kappa$ and $|\cal E|$, such that
  $$\PP\Big[\bigcap_{\xi\in\Xi}\{\gamma\cap \xi\ne\emptyset\}\Big]\le C_{|\cal E|}\prod_{e\in\cal E} \frac{P_{y_e}(r_e)}{P_{y_e}(R_e)}.$$ \label{key-lem2}
\end{Theorem}

\no{\bf Discussion.} Suppose $\gamma$ visits all $\xi\in\Xi$. For $\xi_1,\xi_2\in\Xi$, if $\xi_1<\xi_2$, then $\gamma$ will visit $\xi_1$ before $\xi_2$. Other than these constraints, $\gamma$ can visit the elements in $\Xi$ in any order. The simplest case is that $\gamma$ does not jump back and forth between different groups $\{\Xi_e:e\in\cal E\}$. This means that $\gamma$ first visits all circles in $\Xi_{e_1}$ for some $e_1\in\cal E$ before all other circles in $\Xi$, then visits all circles in $\Xi_{e_2}$ for some $e_2\in{\cal E}\sem\{e_1\}$ before circles in $\Xi\sem(\Xi_{e_1}\cup\Xi_{e_2})$, and so on. In this case, we can easily use the $1$-point estimate and DMP to get the righthand side of the above formula. We use Theorem \ref{key-lem} to deal with the general cases. The key point is that $\gamma$ has to pay a price to  jump back and forth between different $\Xi_e$'s due to the factor $(\frac{r_0}{R_0})^{\alpha/4}$ given in Theorem \ref{key-lem}.

\begin{proof} We write $\N_n$ for $\{k\in\N:k\le n\}$.
  Let $S$ denote the set of bijections $\sigma:\N_{|\Xi|}\to \Xi$ such that $\xi_1<\xi_2$ implies that $\sigma^{-1}(\xi_1)<\sigma^{-1}(\xi_2)$. Let $E=\bigcap_{\xi\in\Xi}\{\gamma\cap \xi\ne\emptyset\}$ and
  $$E^\sigma=\{\tau_{\sigma(1)}<\tau_{\sigma(2)}<\cdots<\tau_{\sigma(|\Xi|)}<\infty\},\quad \sigma\in S.$$
 Then the above discussion gives
  \BGE E=\bigcup_{\sigma\in S} E^\sigma.\label{E-sigma}\EDE
We will derive an upper bound of $\PP[E^\sigma]$ in (\ref{E-sigma-est}).

  Fix $\sigma\in S$. For $e\in\cal E$, we label the elements of $\Xi_e$ by $\xi^e_0<\cdots<\xi^e_{N_e}$, where $N_e=|\Xi_e|-1$.
  Let
  $$J_e=\{1\le n\le {N_e}: \sigma^{-1}(\xi^e_{n})>\sigma^{-1}(\xi^e_{n-1})+1\}\cup\{0\},$$
In plain words,  $n\in J_e$ means that either $n=0$ or after visiting $\xi^e_{n-1}$, $\gamma$ does not immediately visit $\xi^e_n$ without visiting other circles in $\Xi$ that it has not visited before. In the latter case, after visiting $\xi^e_{n-1}$, $\gamma$ visits the circles in $\bigcup_{e'\ne e}\Xi_{e'}$ before $\xi^e_n$.

Order the elements of $J_e$ by $0=s_e(0)<\cdots<s_e(M_e)$, where $M_e=|J_e|-1$. Set $s_e(M_e+1)=N_e+1$.
  Every $\Xi_e$ can be partitioned into $M_e+1$  subsets:
  $$\Xi_{(e,j)}=\{\xi^e_n:s_e(j)\le n\le s_e(j+1)-1\},\quad 0\le j\le M_e.$$
  The meaning of the partition is that, after $\gamma$ visits the first element in $\Xi_{(e,j)}$, which must be $\xi^e_{s_e(j)}$, it then visits all elements in $\Xi_{(e,j)}$ without visiting any other circles in $\Xi$ that it has not visited before.
  Let $I=\{(e,j):e\in{\cal E}, 0\le j\le M_e\}$. Then $\{\Xi_{\iota}:\iota\in I\}$ is another partition of $\Xi$, which is finer than $\{\Xi_e:e\in\cal E\}$. Note that every $\sigma^{-1}(\Xi_\iota)$, $\iota\in I$, is a connected subset of $\Z$.

  For $\iota\in I$, let $e_\iota$ denote the first coordinate of $\iota$, $z_\iota=z_{e_\iota}$ and $y_\iota=\Imm z_\iota$. Let $P_\iota= \frac{P_{y_{\iota}}(R_{\max \Xi_\iota})}{P_{y_{\iota}}(R_{\min \Xi_\iota})}$. Recall that if $\iota=(e,j)$, $\min \Xi_\iota=\xi^e_{s_e(j)}$ and $\max \Xi_\iota=\xi^e_{s_e(j+1)-1}$. From Lemma \ref{one-point-lemma} we get
  \BGE \PP[\tau_{\max\Xi_\iota}<\infty|\F_{\min\Xi_\iota}]\le C P_\iota,\quad \iota\in I.\label{P-iota}\EDE
  Let $P_e= \frac{P_{y_e}(r_e)}{P_{y_e}(R_e)}$, $e\in\cal E$. From (\ref{P-compare}) we get
  \BGE \prod_{j=0}^{M_e} P_{(e,j)}\le 4^{\alpha M_e} P_e,\quad e\in\cal E.\label{P-iota<}\EDE

  We have $|I|=\sum_{e\in\cal E}(M_e+1)$. Considering the order that $\gamma$ visits $\Xi_\iota$, $\iota\in I$, we get a bijection map $\ha\sigma:\N_{|I|}\to  I$ such that $n_1<n_2$ implies that  $\max\sigma^{-1}(\Xi_{\ha\sigma(n_1)})<\min\sigma^{-1}(\Xi_{\ha\sigma(n_2)})$, and $n_1=n_2-1$ implies that $\max\sigma^{-1}(\Xi_{\ha\sigma(n_1)})=\min\sigma^{-1}(\Xi_{\ha\sigma(n_2)})-1$. We may now express $E^\sigma$ as
  $$E^\sigma=\{\tau_{\min\Xi_{\ha\sigma(1)}}<\tau_{\max\Xi_{\ha\sigma(1)}}<\tau_{\min\Xi_{\ha\sigma(2)}}<\tau_{\max\Xi_{\ha\sigma(2)}}<\cdots<
  \tau_{\min\Xi_{\ha\sigma(|I|)}}<\tau_{\max\Xi_{\ha\sigma(|I|)}}<\infty\}.$$

Fix $e_0\in\cal E$. Let $n_j=\ha\sigma^{-1}((e_0,j))$, $0\le j\le M_{e_0}$. Then $n_{j+1}\ge n_{j}+2$, $0\le j\le M_{e_0}-1$. Fix $0\le j\le M_{e_0}-1$. Let $m=n_{j+1}-n_{j}-1$. Applying Theorem \ref{key-lem} to $\ha\xi_0=\min\Xi_{e_0}$, $\xi_0=\max\Xi_{({e_0},j)}=\max\Xi_{\ha\sigma(n_{j})}$, $\xi'_0= \min\Xi_{({e_0},j+1)}=\min\Xi_{\ha\sigma(n_{j+1})}$, $\ha\xi_k=\min\Xi_{\ha\sigma(n_{j}+k)}$ and $\xi_k=\max\Xi_{\ha\sigma(n_j+k)}$, $1\le k\le m$, we get
$$\PP[E^\sigma_{[\max\Xi_{\ha\sigma(n_{j})},\min\Xi_{\ha\sigma(n_{j+1})}]}|\F_{\tau_{\max\Xi_{\ha\sigma(n_{j})}}}]\le C^m 4^{-\alpha/4(s_{e_0}(j+1)-1)} \prod_{n=n_{j}+1}^{n_{j+1}-1}P_{\ha\sigma(n)},$$
where $E^\sigma_{[\max\Xi_{\ha\sigma(n_{j})},\min\Xi_{\ha\sigma(n_{j+1})}]}$ is the $\F_{\tau_{\min\Xi_{\ha\sigma(n_{j+1})}}}$-measurable event
$$\{\tau_{\max\Xi_{\ha\sigma(n_{j})}}<\tau_{\min\Xi_{\ha\sigma(n_{j}+1)}}<\tau_{\max\Xi_{\ha\sigma(n_{j}+1)}}<\cdots
<\tau_{\max\Xi_{\ha\sigma(n_{j}+m)}}<\tau_{\min\Xi_{\ha\sigma(n_{j+1})}}<\infty\}.$$

Letting $j$ vary between $0$ and $M_{e_0}-1$ and using (\ref{P-iota}) and  we get
$$\PP[E^\sigma]\le C^{|I|} 4^{-\alpha/4 \sum_{j=1}^{M_{e_0}} (s_{e_0}(j)-1)} \prod_{\iota\in I} P_\iota.$$
Using (\ref{P-iota<}) and $|I|=\sum_e (M_e+1)$, we find that the right-hand side is bounded by
$$C^{|{\cal E}|} C^{\sum_{e\in\cal E}M_e}4^{-\frac\alpha 4 \sum_{j=1}^{M_{e_0}} s_{e_0}(j)}\prod_{e\in\cal E} P_e,$$
Taking a geometric average over $e_0\in\cal E$, we get
\BGE \PP[E^\sigma]\le   C^{|{\cal E}|} C^{\sum_{e\in\cal E}M_e}4^{-\frac\alpha{4|{\cal E}|}\sum_{e\in\cal E} \sum_{j=1}^{M_{e}} s_{e}(j)}\prod_{e\in\cal E} P_e.\label{E-sigma-est}\EDE

So far we have omitted the $\sigma$ on $I$, $M_e$, $s_e(j)$ and etc; we will put $\sigma$ on the superscript if we want to emphasize the dependence on $\sigma$. From (\ref{E-sigma}) and the above result, it follows that
\BGE \PP[E]\le C^{|{\cal E}|} \sum_{(M_e;(s_e(j))_{j=0}^{M_e})_{e\in\cal E}}|S_{(M_e,(s_e(j)))}|  C^{\sum_{e\in\cal E}M_e}4^{-\frac\alpha{4|{\cal E}|}\sum_{e\in\cal E} \sum_{j=1}^{M_{e}} s_{e}(j)}\prod_{e\in\cal E} P_e,\label{PE-S}\EDE
where
$$S_{(M_e,(s_e(j)))}:=\{\sigma\in S:M^\sigma_e=M_e, s^\sigma_e(j)=s_e(j), 0\le j\le M_e,e\in\cal M\},$$
and the first summation in (\ref{PE-S}) is over all possible $(M_e;(s_e(j))_{j=0}^{M_e})_{e\in\cal E}$, namely, $M_e\ge 0$ and  $0=s_e(0)<s_e(1)<\cdots s_e(M_e)\le N_e$ for every $e\in\cal E$. It now suffices to show that
 \BGE \sum_{(M_e;(s_e(j))_{j=1}^{M_e})_{e\in\cal E}}|S_{(M_e,(s_e(j)))}|  C^{\sum_{e\in\cal E}M_e}4^{-\frac\alpha{4|{\cal E}|}\sum_{e\in\cal E} \sum_{j=1}^{M_{e}} s_{e}(j)} \le C_{|\cal E|},\label{suffice}\EDE
for some $C_{|\cal E|}<\infty$ depending only on $|\cal E|$ and $\kappa$.

We now bound the size of $S_{(M_e,(s_e(j)))}$. Note that $M^\sigma_e$ and $s^\sigma_e(j)$, $0\le j\le M^\sigma_e$, $e\in\cal E$, determine the partition $\Xi_\iota$, $\iota\in I^\sigma$, of $\Xi$. When the partition is given, $\sigma$ is then determined by $\ha\sigma:\N_{|I^\sigma|}\to  I^\sigma$, which is in turn determined by $e_{\ha\sigma(n)}$, $1\le n\le |I^\sigma|=\sum_{e\in\cal E}(M^\sigma_e+1)$, because if $e_{\ha\sigma(n)}=e_0$, then $\ha\sigma(n)=(e_0,j_0)$, where $j_0=\min\{0\le j\le M_{e_0}: (e_0,j)\not\in \ha\sigma(m), m<n\}$. Since each $e_{\ha\sigma(n)}$ has at most $|\cal E|$ possibilities, we have $|S_{(M_e,(s_e(j)))}|\le |{\cal E}|^{\sum_{e\in\cal E}(M_e+1)}$.
Thus, the left-hand side of (\ref{suffice}) is bounded by
$$ |{\cal E}|^{|{\cal E}|} \sum_{(M_e;(s_e(j))_{j=0}^{M_e})_{e\in\cal E}} \prod_{e\in\cal E}   (C|{\cal E}|)^{ M_e}4^{-\frac\alpha{4|{\cal E}|}  \sum_{j=1}^{M_{e}} s_{e}(j)}$$
$$=|{\cal E}|^{|{\cal E}|} \prod_{e\in\cal E} \sum_{M_e=0}^{N_e}  (C|{\cal E}|)^{ M_e} \sum_{0=s_e(0)<\cdots<s_e(M_e)\le N_e}4^{-\frac\alpha{4|{\cal E}|}  \sum_{j=1}^{M_{e}} s_{e}(j)}$$
$$\le |{\cal E}|^{|{\cal E}|} \prod_{e\in\cal E} \sum_{M=0}^{\infty}   (C|{\cal E}|)^{ M} \sum_{s(1)=1}^\infty\cdots  \sum_{s(M)=M}^\infty
4^{-\frac\alpha{4|{\cal E}|}  \sum_{j=1}^{M} s(j)}$$
$$\le  |{\cal E}|^{|{\cal E}|} \prod_{e\in\cal E} \sum_{M=0}^{\infty}   (C|{\cal E}|)^{ M}\prod_{j=1}^M \sum_{s(j)=j}^\infty
4^{-\frac\alpha{4|{\cal E}|}  s(j)}
=  \left[|{\cal E}| \sum_{M=0}^{\infty}   \left(\frac{C|{\cal E}|}{1-4^{-\frac\alpha{4|{\cal E}|} }}\right)^{ M} 4^{-\frac\alpha{8|{\cal E}|}  M(M+1)} \right]^{|{\cal E}|}.$$
The conclusion now follows since the summation inside the square bracket equals to a finite number depending only on $\kappa$ and $|\cal E|$.
\end{proof}

\section{Proofs of the Main Theorems}
First, we are going to use Theorem \ref{key-lem2} to prove Theorem \ref{mainthm2}. What we need to do in the proof is to use the radii $r_j$'s and the distances $l_j$'s to construct a group of circles $\Xi$ and a partition $\Xi_e$, $e\in\cal E$, that satisfy the conditions in Section \ref{Mainsec}, and prove that the upper bound given by Theorem \ref{key-lem2} is comparable to the upper bound in Theorem \ref{mainthm2}.

\begin{proof} [Proof of Theorem \ref{mainthm2}]
  We assume that any $r_j$ is of the form $\frac{l_j}{4^{h_j}}$ for some integer $h_j$. If not, it is between two of them and by changing $C_n$ in the theorem and using (\ref{P-compare}) we can get the result easily. Also we can assume $h_j\ge 1$ for every $j$ because otherwise the corresponding term on right-hand side i.e $\frac{P_{y_j}(r_j\wedge l_j)}{P_{y_j}(l_j)}$ is 1 so we can just ignore it. We want to deduce this theorem from Theorem \ref{key-lem2}, so we want to construct a family  $\Xi$. Consider
\[
\xi_j^s=\{|z-z_j|=\frac{l_j}{4^s}\}, \quad 1\le j\le n,\quad 1\le s\le h_j.
\]
The family $\{\xi_j^s:1\le j\le n,\quad 1\le s\le h_j\}$ may not be mutually disjoint. To solve this issue, we will remove some circles as follows. For $1\le j<k\le n$, let $D_k=\{|z-z_k|\le l_k/4\}$, which contains every $\xi_k^r$, $1\le r\le h_k$, and
\BGE I_{j,k}=\{\xi_j^s: 1\le s\le h_j, \xi_j^s\cap D_k\ne\emptyset\}.\label{I}\EDE
Then $\Xi:=\{\xi_j^s:1\le j\le n, 1\le s\le h_j\}\sem \bigcup_{1\le j<k\le n} I_{j,k}$ is mutually disjoint.
If $\dist(\gamma,z_j)\le r_j$, then $\gamma$ intersects every $\xi_j^s$, $1\le s\le h_j$. So we get
\BGE \PP[\dist(\gamma,z_j)\le r_j,1\le j\le n]\le \PP\Big[\bigcap_{j=1}^n \bigcap_{s=1}^{h_j}\{\gamma\cap \xi_j^s\ne\emptyset\}\Big]
\le \PP\Big[\bigcap_{\xi\in\Xi} \{\gamma\cap\xi\ne\emptyset\}\Big].\label{dist-xi}\EDE

Next, we construct a partition $\{\Xi_e:e\in\cal E\}$ of $\Xi$. First, $\Xi$ has a natural partition $\Xi_j$, $1\le j\le n$, such that $\Xi_j$ is composed of circles centered at $z_j$. For each $j$, we construct a graph $G_j$, whose vertex set is $\Xi_j$, and $\xi_1\ne\xi_2\in \Xi_j$ are connected by an edge iff the bigger radius is $4$ times the smaller one, and the open annulus between them does not contain any other circle in $\Xi$. Let ${\cal E}_j$ denote the set of connected components of $G_j$. Then we partition $\Xi_j$ into $\Xi_e$, $e\in {\cal E}_j$, such that every $\Xi_e$ is the vertex set of $e\in{\cal E}_j$. Then the circles in every $\Xi_e$ are concentric circles with radii forming a geometric sequence with common ratio $1/4$, and the closed annuli $A_e$ associated with $\Xi_e$, $e\in{\cal E}_j$, are mutually disjoint. From the construction we also see that for any $j<k$, and $e\in{\cal E}_j$, $A_e$ does not intersect $D_k$, which contains every $A_e$ with $e\in{\cal E}_k$. Let ${\cal E}=\bigcup_{j=1}^n {\cal E}_j$. Then $A_e$, $e\in\cal E$, are mutually disjoint. Thus, $\{\Xi_e:e\in\cal E\}$ is a partition of $\Xi$ that satisfies the properties before Theorem \ref{key-lem2}. So we get
\BGE\PP\Big[\bigcap_{\xi\in\Xi} \{\gamma\cap\xi\ne\emptyset\}\Big]\le C_{|{\cal E}|} \prod_{e\in\cal E} \frac{P_{y_e}(r_e)}{P_{y_e}(R_e)}=
C_{|{\cal E}|} \prod_{j=1}^n \prod_{e\in {\cal E}_j} \frac{P_{y_j}(r_e)}{P_{y_j}(R_e)}.\label{dist-xi2}\EDE
Here we set $\prod_{e\in{\cal E}_j}=1$ if ${\cal E}_j=\emptyset$. We will finish the proof by comparing $|\cal E|$ with $n$ and the product $\prod_{e\in {\cal E}_j} \frac{P_{y_j}(r_e)}{P_{y_j}(R_e)}$ with $\frac{P_{y_j}(r_j)}{P_{y_j}(R_j)}$.

Here is a useful fact: every $I_{j,k}$ defined in (\ref{I}) contains at most one element. The reason is
$$\frac{\max_{z\in D_k}\{|z-z_j|\}}{\min_{z\in D_k}\{|z-z_j|\}}=\frac{|z_j-z_k|+l_k/4}{|z_j-z_k|-l_k/4}\le \frac{l_k+l_k/4}{l_k-l_k/4}<4.$$
The above formula also implies that, for $j<k$, $\bigcup_{\xi\in\Xi_k}\xi \subset D_k$ intersects at most $2$ annuli from $\{l_j/4^{r}\le |z-z_j|\le l_j/4^{r-1}\}$, $2\le r\le h_j$. If $j>k$, by construction, $\bigcup_{\xi\in\Xi_k}\xi$ is disjoint from the annuli $\{l_j/4^{r}\le |z-z_j|\le l_j/4^{r-1}\}$, $2\le r\le h_j$, which are contained in $D_j$.

We now bound $|{\cal E}_j|$. We may obtain $G$ by removing vertices and edges from a path graph $\ha G_j$, whose vertex set is $\{\xi_j^s: 1\le s\le h_j\}$, and two vertices are connected by an edge iff the bigger radius is $4$ times the smaller one. Every edge $e$ of $\ha G_j$ determines an annulus, denoted by $A_e$. The vertices removed are the elements in $I_{j,k}$, $k>j$; and the edges removed are those $e$ such that $A_e$ intersects some $\xi\in\Xi_k$ with $k\ne j$, which may happen only if $k>j$. Thus, the total number of vertices or edges removed is not bigger than $\sum_{k>j} (1+2)=3(n-j)$. So we get $|{\cal E}_j|\le 1+3(n-j)$. Thus, $|{\cal E}|\le n+\frac{3n(n-1)}{2}$. This means that $C_{|{\cal E}|}$ may be written as $C_n$.

Finally we compare $\prod_{e\in {\cal E}_j} \frac{P_{y_j}(r_e)}{P_{y_j}(R_e)}$ with $\frac{P_{y_j}(r_j)}{P_{y_j}(R_j)}$. If $A$ is an annulus $\{r\le |z-z_0|\le R\}$ for some $z_0\in\lin\HH$ with $y_0\in\Imm z_0\ge 0$ and $R\ge r>0$, we define $P_A=\frac{P_{y_0}(r)}{P_{y_0}(R)}$. Let $A_{j,s}=\{l_j/4^{s}\le |z-z_j|\le l_j/4^{s-1}\}$, $1\le s\le h_j$, and ${\cal S}_j=\{s\in\N_{h_j}:A_{j,s}\subset \bigcup_{e\in\Xi_j} A_e\}$.
 Then
$$\frac{P_{y_j}(r_j)}{P_{y_j}(l_j)}=\prod_{s=1}^{h_j} P_{A_{j,s}},\quad \prod_{e\in {\cal E}_j} \frac{P_{y_j}(r_e)}{P_{y_j}(R_e)}=\prod_{s\in{\cal S}_j} P_{A_{j,s}}.$$
Using (\ref{P-compare}), we get
$$\prod_{e\in {\cal E}_j} \frac{P_{y_j}(r_e)}{P_{y_j}(R_e)}\le 4^{\alpha |\N_{h_j}\sem {\cal S}_j|} \frac{P_{y_j}(r_j)}{P_{y_j}(l_j)}.$$
Now $s\in\N_{h_j}\sem{\cal S}_j$ only if $s=1$ or there is some $k>j$ with $D_k\cap A_{j,s}\ne\emptyset$. Since for $k>j$, $D_k$ intersects at most two $A_{j,s}$, we find that $|\N_{h_j}\sem {\cal S}_j|\le 1+2(n-j)$. Thus,
$$\prod_{j=1}^n \prod_{e\in {\cal E}_j} \frac{P_{y_j}(r_e)}{P_{y_j}(R_e)}\le 4^{\alpha n^2} \prod_{j=1}^n  \frac{P_{y_j}(r_e)}{P_{y_j}(R_e)}.$$
Combining the above formula with (\ref{dist-xi}) and (\ref{dist-xi2}), we complete the proof.
\end{proof}

\begin{proof}[ Proof of Theorem \ref{mainthm}]

As we mentioned before we can define natural length of $SLE$ in a domain by Minkowski content. See equation \eqref{minkcont}. Similarly if $D$ is a bounded subset of the upper half plane we can define $\mcon(\gamma \cap D)$ as the natural length of $SLE$ in the domain $D$ in the obvious way.


The main theorem of \cite{LR} becomes
\[
\lim_{r \rightarrow 0} \mcon(\gamma \cap D;r)=\mcon(\gamma \cap D),
\]
with probability 1. Now we compute $$\EE[\mcon(\gamma \cap D;r)^n]=\EE[r^{n(d-2)} ({\rm Area}(z \in D\;| \dist(z,\gamma)<r)^n)$$
$$=r^{n(d-2)}\EE\Big[\Big(\int_{D} 1_{\dist(z,\gamma)<r}dA(z)\Big)^n\Big]$$
$$= \int_{D^n} r^{n(d-2)}\PP(\dist(z_1,\gamma)<r,...,\dist(z_n,\gamma)<r)dA(z_1)...dA(z_n).$$
For the above equality, we changed expectation and integral which is allowed because the integrand is always positive. We will find an upper bound for
\[
\sup\{r^{n(d-2)}\PP(\dist(z_1,\gamma)<r,...,\dist(z_n,\gamma)<r)\},
\]
which is integrable over $D^n$. By Theorem \ref{mainthm2} we know that this is bounded above by
\[
r^{n(d-2)}C_n\prod_{k=1}^n \frac{P_{y_k}(r\wedge l_k)}{P_{y_k}(l_k)}.
\]
Now assume that $r$ is smaller than $l_{i_1},...,l_{i_k}$ and bigger than the rest. Then by equation (\ref{P-compare}) and the definition of $P_y$ we get that the above quantity is bounded by
\[
C_nr^{n(d-2)}\prod_{j=1}^k \frac{r^{2-d}}{l_{i_j}^{2-d}} \le C_n\prod_{s=1}^n {l_{s}^{d-2}}.
\]
We have the last inequality because if $r>l$ then $r^{d-2} < l^{d-2}$. So now we should show
\[
f(z_1,...,z_n)=\prod_{k=1}^n l_k^{d-2}=\prod_{k=1}^n \min\{|z_k-z_0|,|z_k-z_1|,\dots,|z_k-z_{k-1}|\}^{d-2}
\]
is integrable over $D^n$. This is true because for every $1\le k\le n$,
$$\int_D  \min\{|z_k-z_0|,|z_k-z_1|,\dots,|z_k-z_{k-1}|\}^{d-2}dA(z_k)\le \sum_{j=0}^{k-1} \int_D |z_k-z_j|^{d-2}dA(z_k)$$
$$\le k\int _{|z|\le \diam(D\cup\{0\})} |z|^{d-2}dA(z)=2\pi k\int_0^{\diam(D\cup\{0\})} r^{d-1}dr<\infty,$$
as $d>0$. Finally, we may apply Fatou's lemma to conclude that
$$\EE[\mcon(\gamma \cap D)^n]\le \int_D\,\cdots\,\int_D \prod_{k=1}^n l_k(z_1,\dots,z_n) dA(z_1)\cdots dA(z_n)<\infty.$$
\end{proof}

\end{document}